\documentclass[reqno,12pt]{amsart}
\usepackage{geometry}
\usepackage{mathtools}
\usepackage{hyperref}
\numberwithin{equation}{section}
\newif\ifdoublespace 
\doublespacefalse  
\newtheorem{theorem}{Theorem}[section]
\newtheorem{remark}[theorem]{Remark}
\newtheorem{cor}[theorem]{Corollary}
\newtheorem{prop}[theorem]{Proposition}
\newtheorem{definition}[theorem]{Definition}
\newtheorem{lemma}[theorem]{Lemma}

\DeclareMathOperator{\Tan}{Tan}
\DeclareMathOperator{\nor}{nor}

\DeclareMathOperator{\reg}{reg}
\DeclareMathOperator{\regn}{regn}

\DeclareMathOperator{\ap}{ap}

\newcommand{\R}[1]{\mathbb{R}^{#1}}
\newcommand{\vectornorm}[1]{\left|\left|#1\right|\right|}
\newcommand{\chara}[1]{{\bold{1}_{#1}}}
\newcommand{\mres}{\mathbin{\vrule height 1.6ex depth 0pt width 0.13ex\vrule height 0.13ex depth 0pt width 1.3ex}}
\makeatletter
\begin{document}

\title{On $L_p$ Affine Surface Area and Curvature Measures}

\author{Yiming Zhao}

\address{Department of Mathematics, Polytechnic School of Engineering, New York University, 6 Metrotech Center, Brooklyn, NY 11201, United States}

\email{yiming.zhao.math@gmail.com}

\begin{abstract}{The relationship between $L_p$ affine surface area and curvature measures is investigated. As a result, a new representation of the existing notion of $L_p$ affine surface area depending only on curvature measures is derived. Direct proofs of the equivalence between this new representation and those previously known are provided. The proofs show that the new representation is, in a sense, ``polar'' to that of Lutwak's and ``dual'' to that of Sch\"{u}tt \& Werner's.}
\end{abstract}

\maketitle

\section{Introduction}
\label{sec_intro}
Since its introduction by Lutwak in \cite{lutwak1996}, $L_p$ affine surface area (defined below) has become a fundamental concept in the $L_p$ Brunn--Minkowski theory and has appeared in a growing number of works (e.g., Ludwig \cite{MR1991649,MR2652209}, Paouris \& Werner \cite{MR2880241}, Werner \& Ye \cite{MR2414321,MR2640049}, and Ye \cite{Ye05022014}). Different approaches to $L_p$ affine surface area have been discussed, e.g.,  Meyer \& Werner \cite{MR1764106}, Sch\"{u}tt \& Werner \cite{MR2074173}, and Werner \cite{MR2360611}. Characterization theorems for $L_p$ affine surface area have been given, e.g., Haberl \& Parapatits \cite{MR3194492,MR3176613}, and Ludwig \& Reitzner \cite{MR2680490}.  Its relation to PDE was explored, e.g., Lutwak \& Oliker \cite{MR1316557}. Connections between $L_p$ affine surface area and information theory were discovered in Werner \cite{MR2921171}. For the $p=1$ case, $L_p$ affine surface area is an older notion usually referred to simply as affine surface area. Results here are even more numerous. Different approaches to this notion include Leichtwei{\ss} \cite{leichtweiss19862}, Lutwak \cite{MR896173,lutwak1991}, Meyer \& Werner \cite{MR1466952}, Sch\"{u}tt \& Werner \cite{schutt1990}, and Werner \cite{MR1292847,MR1669674}.  A characterization of affine surface area was given in Ludwig \& Reitzner \cite{ludwig}. Connections between affine surface area and the affine Plateau problem were discussed, e.g., Trudinger \& Wang \cite{MR2137978}. It is also not surprising to see the appearance of affine surface area in polytopal approximation, e.g., B\'{a}r\'{a}ny \cite{MR1437299}, B\"{o}r\"{o}czky \cite{MR1770932}, and Gruber \cite{MR731110,MR1242984,MR1654596}. Perhaps, most importantly, it is the crucial ingredient in fundamental affine isoperimetric inequalities and gives rise to some affine analytic inequalities, e.g., Artstein-Avidan, Klartag, Sch\"{u}tt \& Werner \cite{MR2899992}, Caglar \& Werner \cite{MR3187648}, and Lutwak \cite{MR840399,MR1242979}.

Let $K\subset \R{n}$ be a convex body (compact convex set with non-empty interior). The curvature measures of $K$ are a list of $n$ Borel measures defined on the boundary of $K$ that can be defined via local parallel sets (see e.g., Chapter 4 in \cite{schneider1993}). In the current paper, the relationship between $L_p$ affine surface area and curvature measures will be explored. As a result, a new representation of the existing notion of $L_p$ affine surface area using only curvature measures will be derived. New proofs of the important properties of $L_p$ affine surface area, such as the upper semi-continuity and the $L_p$ affine isoperimetric inequality, will be given using the new representation. The proofs given will be ones requiring no prior knowledge of the properties already established using other definitions of $L_p$ affine surface area. It is also the aim of this paper to investigate the relationship between the new representation of $L_p$ affine surface area and the three existing ones. This will be done by providing direct proofs of equivalence between the new representation and those previously given. It will become apparent that the new form of $L_p$ affine surface area is, in a sense, ``polar'' to that of Lutwak's and ``dual'' to Sch\"{u}tt \& Werner's. In order to establish the equivalence, the Lipschitz property of a sequence of restrictions of the inverse Gauss map will be investigated, which may be of independent interest. This will be discussed in Section \ref{sec_lip}.

It is important to note that it is the attempt of the current paper to give a new representation of the usual $L_p$ affine surface area (not to define a new (different) $L_p$ affine surface area).

The notion of affine surface area traces back to affine differential geometry. In affine differential geometry, the affine surface area of a convex body $K$ with sufficiently smooth boundary (at least $C^2$) and everywhere positive Gauss curvature is given by 
\begin{equation}
\label{eq_intro1}
\Omega (K)=\int_{\partial K} H_K^{\frac{1}{n+1}}(x)d\mathcal{H}^{n-1}(x),
\end{equation}
where $H_K(x)$ is the Gauss curvature of $K$ at $x\in \partial K$ (the boundary of $K$) and $\mathcal{H}^{n-1}$ is $(n-1)$ dimensional Hausdorff measure.

When $K$ has $C^2$ boundary with positive Gauss curvature, the Gauss map $\nu_K: \partial K\rightarrow S^{n-1}$ is nice enough to allow the change of variable $u=\nu_K(x)$ and we get:
\begin{equation}
\label{eq_intro2}
\Omega(K)=\int_{S^{n-1}}F_K^{\frac{n}{n+1}}(u)d\mathcal{H}^{n-1}(u).
\end{equation}
Here $F_K:S^{n-1}\rightarrow \mathbb{R}$ is the curvature function of $K$.

A very important result in affine differential geometry is the affine isoperimetric inequality which characterizes ellipsoids. For a convex body $K\subset \R{n}$ with $C^2$ boundary and positive Gauss curvature,
\begin{equation}
\label{star}
\Omega(K)^{n+1}\leq n^{n+1}\omega_n^2 V(K)^{n-1},
\end{equation}
with equality if and only if $K$ is an ellipsoid. Here $V(K)$ is the volume of $K$ and $\omega_n$ is the volume of the $n$-dimensional unit ball.

Extending the definition of affine surface area to one that works for general convex bodies (without smoothness assumptions) and still respects the basic properties of the classical definition was of huge interest during the late 80s and 90s (in the previous century). In particular, is there a way to do the extension so that the affine isoperimetric inequality, with the same equality conditions, still holds? The first attempt was made by Petty \cite{MR809198}. He observed that \eqref{eq_intro2} makes sense for convex bodies that possess curvature functions. The affine isoperimetric inequality was also shown to  hold under this extension. 

Although the Gauss curvature and the curvature function do not necessarily exist for general convex bodies, the generalized Gauss curvature and the generalized curvature function (see \cite{schneider1993} or Section \ref{section_preli}) exist almost everywhere on $\partial K$ and $S^{n-1}$ with respect to $(n-1)$ dimensional Hausdorff measure. Thus \eqref{eq_intro1} and \eqref{eq_intro2} already suggest two possible extensions. But, the two extensions are not trivial at all, since the two integrals might not make sense. 

With the notion of floating body,  Leichtwei{\ss} was able to give a geometric meaning to \eqref{eq_intro2} when $F_K$ is the generalized curvature function. 
\begin{definition}[Affine surface area by Leichtwei{\ss} \cite{leichtweiss19862}]Let $K\subset \R{n}$ be a convex body. The affine surface area of $K$ is given by
\begin{equation}
\label{eq_new_1}
\Omega(K)=\int_{S^{n-1}}F_K^{\frac{n}{n+1}}(u)d\mathcal{H}^{n-1}(u),
\end{equation}
where $F_K$ is the generalized curvature function.
\end{definition}

One is tempted to use a strategy similar to how we arrived at \eqref{eq_intro2} to get \eqref{eq_intro1} to work for general convex bodies. This turns out to be invalid since neither the Gauss map nor the inverse Gauss map, in this case, is smooth enough (in fact, not even Lipschitz) to permit such a change of variable. In spite of this unfortunate fact, there is a natural extension to \eqref{eq_intro1}. Sch{\"u}tt \& Werner \cite{schutt1990}, via the notion of convex floating body, were able to give a geometric meaning to the integral representation \eqref{eq_intro1} with $H_K$ being the generalized Gauss curvature. 

\begin{definition}[Affine surface area by Sch{\"u}tt \& Werner \cite{schutt1990}]Let $K\subset \R{n}$ be a convex body. The affine surface area of $K$ is given by 
\begin{equation}
\label{eq_new_2}
\Omega(K)=\int_{\partial K}H_K^{\frac{1}{n+1}}(x)d\mathcal{H}^{n-1}(x),
\end{equation}
where $H_K$ is the generalized Gauss curvature. 
\end{definition}

One of the major characteristics that distinguishes affine surface area (look at either \eqref{eq_new_1} or \eqref{eq_new_2}) and other geometric invariants is that affine surface area is not continuous with respect to the Hausdorff metric. For example, any convex body can be approximated by polytopes, but polytopes are always of zero affine surface area. Given this fact, only upper semi-continuity can be expected. But, even establishing the upper semi-continuity of classical affine surface area (in the smooth case) was unsolved in the 80s. One of the difficulties of establishing this lies in the lack of knowledge of the limit behaviors of the Gauss curvature and the curvature function. Since \eqref{eq_new_1} and \eqref{eq_new_2} take similar formulations, difficulty persists. Note that the lack of continuity also adds to the difficulty of establishing the affine isoperimetric inequality, since we cannot establish the inequality for a dense class of convex bodies and then take a limit. 

The long conjectured upper semi-continuity of classical affine surface area was settled by Lutwak. In \cite{lutwak1991}, he found the following characterization of affine surface area for general convex bodies.
\begin{definition}[Affine surface area by Lutwak \cite{lutwak1991}]Let $K\subset \R{n}$ be a convex body. The affine surface area of $K$ is given by
\begin{equation}
\label{eq_new_3}
\Omega(K)=\inf_{h}\left\{\left(\int_{S^{n-1}}h^{n}(u)d\mathcal{H}^{n-1}(u)\right)^{\frac{1}{n+1}}\left(\int_{S^{n-1}}h^{-1}(u)dS_K(u)\right)^{\frac{n}{n+1}}\right\},
\end{equation}
where the infimum is taken over all positive continuous functions $h: S^{n-1}\rightarrow \mathbb{R}$, and $S_K$ is the surface area measure of $K$.
\end{definition}
Note that Lutwak's definition applies to all convex bodies (even the ones without smoothness assumptions). More importantly, the upper semi-continuity of \eqref{eq_new_3} follows directly from the weak continuity of the measures in the integrals and the fact that the infimum of a class of continuous functionals is upper semi-continuous. Since \eqref{eq_new_3} agrees with classical affine surface area for smooth convex bodies with everywhere positive curvature function (also shown in \cite{lutwak1991}), this, in turn, proves the upper semi-continuity of classical affine surface area. As shown in \cite{lutwak1991}, the affine isoperimetric inequality can also be established in this case by using the Blaschke--Santal\'{o} inequality.

A natural question to ask is: are Definitions \eqref{eq_new_1}, \eqref{eq_new_2}, and \eqref{eq_new_3} equivalent? As was explained earlier, the equivalence of \eqref{eq_new_1} and \eqref{eq_new_2} is by no means a trivial problem for general convex bodies. It was not until 1993 that Sch{\"u}tt, in \cite{schutt1993}, proved the equivalence---using a somewhat indirect method. A direct proof was given in \cite{daniel1996} by Hug. That \eqref{eq_new_1} and \eqref{eq_new_3} are equivalent was shown by Leitchwei{\ss} in \cite{leitchtweiss1989}. Note that in Lutwak's original definition, the infimum is only taken over all positive functions $h$ such that $\int_{S^{n-1}}uh(u)^{n+1}d\mathcal{H}^{n-1}(u)=o$. The removal of this restriction was proposed by Leichtwei{\ss} in \cite{leitchtweiss1989} and the equivalence between this formulation and Lutwak's original definition was shown by Dolzmann \& Hug in \cite{daniel1995} using a topological argument.

Note that it is trivial to see that affine surface area is translation invariant. Sch{\"u}tt, in \cite{schutt1993}, proved that affine surface area is a valuation. Hence, affine surface area is an upper semi-continuous valuation that is invariant under volume preserving affine transformations. In a landmark work, Ludwig \& Reitzner \cite{ludwig} established the ``converse'': if a real-valued upper semi-continuous valuation on the set of convex bodies is invariant under volume preserving affine transformations, then it must be of the form $c_0V_0+c_1V+c_2\Omega$ with $c_0,c_1,c_2\in \mathbb{R}$ and $c_2>0$. Here $V_0$ is the Euler characteristic, $V$ is volume, and $\Omega$ is affine surface area.

Observe that \eqref{eq_new_1} and \eqref{eq_new_2} are ``polar'' to each other in the sense that one is defined as an integral over the boundary of the convex body (domain of the Gauss map), while the other is defined as an integral over the unit sphere (image of the Gauss map). In fact, as the proof provided by Hug in \cite{daniel1996} indicates, Definitions \eqref{eq_new_1} and \eqref{eq_new_2} are linked by the Gauss map. Also note that \eqref{eq_new_3} is ``dual'' to \eqref{eq_new_1} as one can see in the proof in \cite{leitchtweiss1989}.

Recall that for a convex body $K\subset \R{n}$, curvature measures are a list of $n$ Borel measures defined on $\partial K$ that can be defined via local parallel sets. For details, the reader should consult Chapter 4 in \cite{schneider1993}. Among these curvature measures, the $0$-th and $(n-1)$-th curvature measures have stronger geometric meanings. More specifically, for a Borel set $\beta\subset \partial K$, the $0$-th curvature measure $C_0(K,\beta)$ and the $(n-1)$-th curvature measure $C_{n-1}(K,\beta)$ of $\beta$ are given by 
\begin{equation*}
C_{0}(K,\beta)=\mathcal{H}^{n-1}(\nu(K,\beta))\quad \text{and}\quad C_{n-1}(K,\beta)=\mathcal{H}^{n-1}(\beta),
\end{equation*}
where $\nu(K,\beta)\subset S^{n-1}$ is the set of outer unit normals of $K$ at points in $\beta$.

Note that previous formulations of affine surface area involve $(n-1)$ dimensional Hausdorff measure on $\partial K$ (Definition \eqref{eq_new_2}), $(n-1)$ dimensional Hausdorff measure on $S^{n-1}$ (Definitions \eqref{eq_new_1} and \eqref{eq_new_3}), and the surface area measure of $K$ (Definition \eqref{eq_new_3}). Since curvature measures are also a crucial type of measures associated to a convex body $K$, it is natural to study the relationship between affine surface area and the curvature measures. It is the purpose of this paper to investigate this missing element. To be precise, the following theorem will be proved:
\begin{theorem}
\label{thm_new_1}
For a convex body $K\subset \R{n}$,
\begin{equation}
\label{eq_new_4}
\int_{\partial K}H_K^{\frac{1}{n+1}}(x)d\mathcal{H}^{n-1}(x)=\inf_g\left\{\left(\int_{\partial K}g^{-n}(x)dC_0(K,x)\right)^{\frac{1}{n+1}}\left(\int_{\partial K} g(x)dC_{n-1}(K,x)\right)^{\frac{n}{n+1}}\right\},
\end{equation}
where the infimum is taken over all positive continuous functions $g:\partial K\rightarrow \mathbb{R}$.
\end{theorem}

In light of Theorem \ref{thm_new_1}, we may view the right side of \eqref{eq_new_4} as a new representation of the existing notion of affine surface area. 
\begin{definition}Let $K\subset \R{n}$ be a convex body. The affine surface area of $K$ can be defined by
\begin{equation}
\label{eq_new_5}
\Omega(K)=\inf_{g}\left\{\left(\int_{\partial K}g^{-n}(x)dC_0(K,x)\right)^{\frac{1}{n+1}}\left(\int_{\partial K} g(x)dC_{n-1}(K,x)\right)^{\frac{n}{n+1}}\right\},
\end{equation}
where the infimum is taken over all positive continuous functions $g:\partial K\rightarrow \mathbb{R}$.
\end{definition}
Note again that this form of affine surface area uses only curvature measures.

A list of properties of \eqref{eq_new_5} will be given in Section \ref{sec_properties}. In particular, among other things, the upper semi-continuity (Theorem \ref{prop_new_1}) and the affine isoperimetric inequality with equality condition (Theorem \ref{prop_new_2}) will be demonstrated using the new representation \eqref{eq_new_5}. The author would like to point out that although previously established, none of the results proved in Section \ref{sec_properties} require acknowledgement of properties proved under the existing forms.

With the recent development of the $L_p$ Brunn-Minkowski theory, efforts were also made to generalize affine surface area to its $L_p$ analogue. One of the key findings in the $L_p$ Brunn-Minkowski theory is the $L_p$ curvature function discovered by Lutwak \cite{MR1231704,lutwak1996}. Given a convex body $K$ that possesses a curvature function $F_K$, the $L_p$ curvature function may be defined by $h_K^{1-p}F_K$ with $h_K$ being the support function of $K$. Since the generalized curvature function exists almost everywhere for an arbitrary convex body, the generalized $L_p$ curvature function also exists almost everywhere. Given this notion, $L_p$ affine surface area can be defined.

The notion of $L_p$ affine surface area was introduced by Lutwak in \cite{lutwak1996}. He defined the $L_p$ affine surface area of a convex body $K\subset \R{n}$ that contains the origin in its interior to be,
\begin{equation}
\label{eq_new_6}
\Omega_p(K)=\inf_{h}\left\{\left(\int_{S^{n-1}}h^{n}(u)d\mathcal{H}^{n-1}(u)\right)^{\frac{p}{n+p}}\left(\int_{S^{n-1}}h^{-p}(u)h_K^{1-p}(u)dS_K(u)\right)^{\frac{n}{n+p}}\right\},
\end{equation}
where the infimum is taken over all positive continuous functions $h: S^{n-1}\rightarrow \mathbb{R}$. Although Lutwak originally presented this definition for the case $p\geq 1$, it works perfectly fine for any $0<p<1$ as observed by Hug in \cite{daniel1996}.

As with the classical $p=1$ case, different forms of $L_p$ affine surface area exist.  For each $p>0$, the following form of  $L_p$ affine surface area was given by Lutwak \cite{lutwak1996} for convex bodies that possess a continuous curvature function and by Hug \cite{daniel1996} for general convex bodies, 
\begin{equation}
\label{eq_new_7}
\Omega_p(K)=\int_{S^{n-1}}\left(\frac{F_K(u)}{h_K^{p-1}(u)}\right)^{\frac{n}{n+p}}d\mathcal{H}^{n-1}(u).
\end{equation} 
Analogously, for each $p>0$, another form of $L_p$ affine surface area, was given by Hug in \cite{daniel1996},
\begin{equation}
\label{eq_new_8}
\Omega_p(K)=\int_{\partial K}\left(\frac{H_K(x)}{(h_K(\nu_K(x)))^{(p-1)n/p}}\right)^{\frac{p}{n+p}}d\mathcal{H}^{n-1}(x).
\end{equation}

Note that \eqref{eq_new_7} is the $L_p$ extension of \eqref{eq_new_1} while \eqref{eq_new_8} is the $L_p$ extension of \eqref{eq_new_2}.

The equivalence of \eqref{eq_new_7} and \eqref{eq_new_8} was proved by Hug in \cite{daniel1996}. That \eqref{eq_new_6} and \eqref{eq_new_7} are equivalent for convex bodies that possess a positive continuous curvature function was due to Lutwak \cite{lutwak1996}, and can in general be proved in a similar way as Leitchwei{\ss} did in \cite{leitchtweiss1989} for $p=1$ as pointed out in \cite{daniel1996}.

It is also possible to prove the analogue of Theorem \ref{thm_new_1} in the $L_p$ setting. Namely,
\begin{theorem}
\label{thm_new_2}
Let $p>0$ be a real number. For each convex body $K\subset \R{n}$ that contains the origin in its interior, 
\begin{equation*}
\begin{aligned}
&\int_{\partial K}\left(\frac{H_K(x)}{(h_K(\nu_K(x)))^{(p-1)n/p}}\right)^{\frac{p}{n+p}}d\mathcal{H}^{n-1}(x)\\
=&\inf_{g}\left\{\left(\int_{\partial K}g^{-n}(x)dC_0(K,x)\right)^{\frac{p}{n+p}}\left(\int_{\partial K} g^p(x)(h_K(\nu_K(x)))^{1-p}dC_{n-1}(K,x)\right)^{\frac{n}{n+p}}\right\},
\end{aligned}
\end{equation*}
where the infimum is taken over all positive continuous functions $g:\partial K\rightarrow \mathbb{R}$.
\end{theorem}

\section{Preliminaries}
\label{section_preli}
We will be working mainly in $\R{n}$ with the canonical inner product $\langle\cdot,\cdot\rangle$. The usual Euclidean $2$-norm will be denoted by $\vectornorm{\cdot}$ and the open (resp. closed) ball of radius $r$, which is centered at $x$, will be denoted by $B(x,r)$ (resp. $B[x,r]$). We write $\omega_n$ for the volume of the $n$-dimensional unit ball. For a subset $A\subset \R{n}$, we will write $\bar{A}$ and $A^c$ for the closure of $A$ and the complement of $A$, respectively. The characteristic function of $E$, for any set $E$, is written as $\chara{E}$. 

A subset $K$ of $\R{n}$ is called a \emph{convex body} if it is a compact convex set with non-empty interior. The set of all convex bodies that contain the origin in the interior is denoted by $\mathcal{K}_0^n$. The boundary of $K$ will be written as $\partial K$. For an integer $m\leq n$, we will write $\mathcal{H}^{m}$ for $m$ dimensional Hausdorff measure. If $\eta$ is a measure on a topological space $X$ and $A\subset X$ is $\eta$ measurable, the restriction of $\eta$ to $A$ will be denoted by $\eta\mres A$. 

Associated to each convex body $K$ is the support function $h_K:\R{n}\rightarrow \mathbb{R}$ given by
\begin{equation*}
h_K(x)=\max\{\langle x,y\rangle:y\in K\},
\end{equation*}
for each $x\in \R{n}$.

The supporting hyperplane $P(K,u)$ of $K$ for each $u\in S^{n-1}$ is given by
\begin{equation*}
P(K,u)=\{x\in \R{n}:\langle x,u\rangle =h_K(u)\}.
\end{equation*}

At each boundary point $x\in \partial K$, a unit vector $u$ is said to be an \emph{outer unit normal} of $K$ at $x$ if $P(K,u)$ passes through $x$. For a subset $\beta\subset \partial K$, the spherical image, $\nu(K,\beta)$, of $K$ at $\beta$, is the set of all outer unit normal vectors of $K$ at points in $\beta$. A boundary point $x$ is \emph{regular} if $\nu(K,\{x\})$ contains exactly one point in $S^{n-1}$. Denote by $\reg K$ the set of all regular boundary points of $K$. The \emph{Gauss map}, $\nu_K:\reg K\rightarrow S^{n-1}$ is the map that takes each regular boundary point to the unique outer unit normal of that point. Similarly, for each subset $\omega\subset S^{n-1}$, the inverse spherical image, $\tau(K,\omega)$, of $K$ at $\omega$, is the set of all boundary points of $K$ that have outer normal vectors in $\omega$. A unit vector $u$ is \emph{regular} if $\tau(K,\{u\})$ contains exactly one point in $\partial K$. Denote by $\regn K$ the set of all regular normal vectors. The \emph{inverse Gauss map}, $\tau_K:\regn K\rightarrow\partial K$ is the map that takes each regular normal vector to the unique point in $\tau(K,\{u\})$. Both the Gauss map and the inverse Gauss map are continuous (see Lemma 2.2.12 in \cite{schneider1993}).

Let $f:\R{n}\rightarrow \mathbb{R}$ be a convex function. The set
\begin{equation*}
\partial f(x)=\{v\in \R{n}:f(y)\geq f(x)+\langle v,y-x\rangle,\,\,\forall y\in \R{n}\}
\end{equation*}  
is called the \emph{subdifferential} of $f$ at $x$. If $\vartheta:\R{n}\rightarrow\R{n}$ satisfies $\vartheta(x)\in\partial f(x)$ for each $x\in \R{n}$, then it is called a \emph{subgradient choice} of $f$. Moreover, $f$ is differentiable at $x$ if and only if $\partial f(x)$ contains only $\nabla f(x)$, the gradient of $f$ at $x$.

The following notion of second order differentiability is useful. We say $f$ is \emph{second order differentiable} at $x_0$ in the generalized sense if $f$ is differentiable at $x_0$ in the classical sense and there exists a symmetric linear map $Af(x_0):\R{n}\rightarrow \R{n}$ such that
\begin{equation*}
f(y)=f(x_0)+\langle\nabla f(x_0),y-x_0\rangle+\frac{1}{2}\langle Af(x_0)(y-x_0),y-x_0\rangle+o(\vectornorm{y-x_0}^2),
\end{equation*}
for every $y\in \R{n}$. It follows from \cite{alexandroff} that a convex function $f$ is \emph{second order differentiable} at $x_0$ in the generalized sense if and only if there exists a neighborhood $V$ of $x_0$ and a symmetric linear map $Af(x_0):\R{n}\rightarrow\R{n}$ such that
\begin{equation*}
\vectornorm{\vartheta(y)-\vartheta(x_0)-Af(x_0)(y-x_0)}= o(\vectornorm{y-x_0})
\end{equation*} 
for all $y\in V$ and all subgradient choices $\vartheta$. Note that the generalized second order differentiability is a local property. Hence, the above notion extends naturally to the case where $f$ is only defined on an open subset of $\R{n}$.
\ifdoublespace
\\
\indent
\else

\fi
For a regular boundary point $x_0\in \partial K$, suppose $u_0=\nu_K(x_0)$. The \emph{tangent space}, $T_{x_0}K$, of $K$ at $x_0$ is the linear subspace $P(K,u_0)-x_0$. Write $\tilde{y}$ for the orthogonal projection of $y$ to $T_{x_0}K$ for each $y\in \R{n}$. We can choose a number $\varepsilon>0$ and a neighborhood $U(x_0)$ of $x_0$ such that for each $x\in U(x_0)\cap \partial K$,
\begin{equation}
\label{eq_local 31}
x=x_0+\tilde{x}-\widetilde{x_0}-f(\tilde{x}-\widetilde{x_0})u_0, 
\end{equation}
where $\vectornorm{\tilde{x}-\widetilde{x_0}}<\varepsilon$ and $f:T_{x_0}K\cap B(o,\varepsilon)\rightarrow \mathbb{R}$ is a convex function satisfying $f\geq 0$ and $f(o)=0$. We say a regular boundary point $x_0$ is \emph{normal} if $f$ in \eqref{eq_local 31} is second order differentiable at $o$ in the generalized sense. Denote by $\nor K$ the set of all normal boundary points of $K$. With a proper choice of orthonormal basis $\mathfrak{B}=\{e_1,e_2,\ldots,e_n\}$ satisfying $e_1,\ldots,e_{n-1}\in T_{x_0}K$ and $e_n=-u_0$, it is possible to write $f$ as:
\begin{equation}
\label{eq_normal representation}
f(\tilde{x}-\widetilde{x_0})=\frac{1}{2}\kappa_1(x_0)(x^1-x_0^1)^2+\ldots+\frac{1}{2}\kappa_{n-1}(x_0)(x^{n-1}-x_0^{n-1})^2+o(\vectornorm{\widetilde{x}-\widetilde{x_0}}^2),
\end{equation}
where $(x^1,\ldots,x^n)$ are the coordinates of $x$ under $\mathfrak{B}$. Here, $\kappa_i(x_0)$ is called a generalized principal curvature while $e_i(x_0)$ is the associated generalized principal direction,  for $1\leq i\leq n-1$. In this case, the generalized Gauss curvature $H_K(x_0)$ of $K$ at $x_0$ is given by
\begin{equation*}
H_K(x_0)=\kappa_1(x_0)\kappa_2(x_0)\cdots\kappa_{n-1}(x_0).
\end{equation*}
It follows from the Alexandrov Theorem \cite{alexandroff,MR1555408} that 
\begin{equation}
\label{sharp}
\mathcal{H}^{n-1}(\partial K\setminus \nor K)=0. 
\end{equation}
Hence, $H_K(x)$ is defined for $\mathcal{H}^{n-1}$ almost all $x\in \partial K$. The set $H^+$ is given by
\begin{equation}
\label{astr}
H^+=\{x\in \partial K:x\text{ is a normal boundary point and }H_K(x)>0\}.
\end{equation}

The support function $h_K$ is differentiable at $u_0\in S^{n-1}$ if and only if $u_0$ is a regular normal vector. In this case, $\nabla h_K(u_0)=\tau_K(u_0)$. (See Corollary 1.7.3 in \cite{schneider1993}.) It was the result of the Alexandrov Theorem \cite{alexandroff,MR1555408} that every convex function $f:\R{n}\rightarrow \mathbb{R}$ is $\mathcal{H}^n$ almost everywhere second order differentiable in the generalized sense. In particular, the support function $h_K$ is $\mathcal{H}^{n}$ almost everywhere second order differentiable in the generalized sense. Denote by $D^2(h_K)$ the set of all points at which $h_K$ is second order differentiable in the generalized sense. The following properties can be easily seen from the homogeneity of $h_K$:
\begin{enumerate}
\item If $u_0\in D^2(h_K)$, then $tu_0\in D^2(h_K)$, for all $t\in \mathbb{R}\setminus\{0\}$. Hence
\begin{equation}
\label{hkdifferentiable}
\mathcal{H}^{n-1}(S^{n-1}\setminus D^2(h_K))=0.
\end{equation}
\item Given $u_0\in D^2(h_K)$, we have that $u_0$ is an eigenvector of $Ah_K(u_0)$ with 0 being the associated eigenvalue. The fact that $Ah_K(u_0)$ is symmetric tells us that $u_0^\perp$ is an invariant subspace of $Ah_K(u_0)$.
\end{enumerate}

The generalized curvature function of $K$ at $u\in D^2(h_K)\cap S^{n-1}$, denoted by $F_K(u)$, is defined to be the determinant of $Ah_K(u)|_{u^\perp}$. Note that $F_K(u)$ is defined for $\mathcal{H}^{n-1}$ almost all $u\in S^{n-1}$. The set $F^+$ is given by 
\begin{equation}
\label{105}
F^+=\{u\in D^2(h_K)\cap S^{n-1}:F_K(u)>0\}.
\end{equation} 

The surface area measure $S_K$ of a convex body $K$ is a Borel measure on $S^{n-1}$ and is given by
\begin{equation*}
S_K(\omega)=\mathcal{H}^{n-1}(\tau(K,\omega)),
\end{equation*}
for each Borel set $\omega\subset S^{n-1}$.

Recall that the $0$-th curvature measure $C_0(K,\cdot)$ and the $(n-1)$-th curvature measure $C_{n-1}(K,\cdot)$ are Borel measures on the boundary of $K$ and are given by
\begin{equation}
\label{eq_new 9}
C_0(K,\beta)=\mathcal{H}^{n-1}(\nu(K,\beta))\qquad \text{and} \qquad C_{n-1}(K,\beta)=\mathcal{H}^{n-1}(\beta),
\end{equation}
for each Borel set $\beta\subset \partial K$. It is obvious that $C_0(K,\cdot)$ and $C_{n-1}(K,\cdot)$ are finite measures. The $0$-th curvature measure $C_0(K,\cdot)$ has the following decomposition (see e.g., Hilfssatz 3.6 in \cite{MR522031} or (2.7) in \cite{MR1654685}): for each Borel set $\beta\subset \partial K$,
\begin{equation}
\label{decom}
C_0(K,\beta)=\int_{\beta}H_K(x)d\mathcal{H}^{n-1}(x)+\int_{\beta\cap \hat{\partial}K}dC_0(K,x),
\end{equation}
where $\hat{\partial}K\subset \partial K$ is a Borel set and $\mathcal{H}^{n-1}(\hat{\partial}K)=0$. In particular, one has 
\begin{equation}
\label{tmp1}
\int_{\partial K}H_K(x)d\mathcal{H}^{n-1}(x)<\infty.
\end{equation}

Curvature measures are weakly continuous with respect to the Hausdorff metric (see \cite{schneider1993}).

The following definitions are needed for Federer's coarea formula. See~\cite{federer} for details.

A subset $\omega$ of $\R{n}$ is said to be \emph{$(\mathcal{H}^{n-1},n-1)$ rectifiable} if $\mathcal{H}^{n-1}(\omega)<\infty$ and there exists $\{(f_i,E_i)\}_{i\in\mathbb{N}^+}$ such that $E_i\subset \R{n-1}$ is bounded, $f_i:E_i\rightarrow \R{n}$ is Lipschitz and $\mathcal{H}^{n-1}(\omega\setminus \cup_{i\in\mathbb{N}^+}f_i(E_i))=0$. 

Let $S$ be a non-empty subset of $\R{n}$. The \emph{tangent cone} of $S$ at a given point $a\in \R{n}$, denoted by $\Tan(S,a)$, can be defined as the set of $v\in \R{n}$ such that for every $\varepsilon>0$ there exists $x\in S$ and $r>0$ with $\vectornorm{x-a}<\varepsilon$ and $\vectornorm{r(x-a)-v}<\varepsilon$. 

Suppose $\eta$ is a measure on $\R{n}$. The \emph{$(n-1)$ dimensional density} $\Theta^{n-1}(\eta,a)$ at $a\in \R{n}$ is given by
\begin{equation*}
\Theta^{n-1}(\eta,a)=\lim_{r\rightarrow 0+}\omega_{n-1}^{-1}r^{-(n-1)}\eta(B(a,r)),
\end{equation*}
if the limit exists. The \emph{$(\eta, n-1)$ approximate tangent cone} $\Tan^{n-1}(\eta, a)$ at $a$ is given by
\begin{equation*}
\Tan^{n-1}(\eta,a)=\bigcap\{\Tan(S,a):S\subset \R{n},\Theta^{n-1}(\eta\mres (\R{n}\setminus S),a)=0\}.
\end{equation*}
Suppose $f$ maps a subset of $\R{n}$ into $\R{n}$. We say that $f$ is \emph{$(\eta,n-1)$ approximately differentiable} at $a$ if and only if there exists $\xi\in \R{n}$ and a continuous linear map $\zeta:\R{n}\rightarrow\R{n}$ such that
\begin{equation*}
\Theta^{n-1}\left(\eta\mres (\R{n}\setminus \{x:\vectornorm{f(x)-\xi-\zeta(x-a)}\leq \varepsilon\vectornorm{x-a}\}),a\right)=0,
\end{equation*}
for every $\varepsilon>0$. In this case, the \emph{$(\eta,n-1)$ approximate differential} of $f$ at $a$, denoted by $(\eta,n-1)\ap\,Df(a)$, is given by
\begin{equation*}
(\eta,n-1)\ap\,Df(a)=\zeta|_{\Tan^{n-1}(\eta,a)}.
\end{equation*}

Suppose $V, W$ are two $(n-1)$-dimensional Hilbert spaces. Let $\bigwedge^{n-1} V$ and $\bigwedge^{n-1} W$ be the $(n-1)$th exterior power of $V$ and $W$ equipped with the induced inner products from $V$ and $W$ respectively. Every linear map $f:V\rightarrow W$ induces a map $\bigwedge^{n-1} f: \bigwedge^{n-1} V\rightarrow \bigwedge^{n-1}W$. By $\vectornorm{\bigwedge^{n-1}f}$, we mean the operator norm of $\bigwedge^{n-1} f$. Note that $\vectornorm{\bigwedge^{n-1}f}$ is just the absolute value of the determinant of $f$ when $V=W$. See Chapter 1 in \cite{federer} for details. 

When $\eta$ is the restriction of $\mathcal{H}^{n-1}$ to some $\mathcal{H}^{n-1}$ measurable and $(\mathcal{H}^{n-1},n-1)$ rectifiable subset of $\R{n}$, by Theorem 3.2.19 in \cite{federer}, the approximate tangent cone $\Tan^{n-1}(\eta, a)$ is an $(n-1)$ dimensional subspace of $\R{n}$ for $\mathcal{H}^{n-1}$ almost all $a$ in that subset. In this case (when $\Tan^{n-1}(\eta,a)$ is an $(n-1)$ dimensional subspace of $\R{n}$), we call $\vectornorm{\bigwedge^{n-1}(\eta,n-1)\ap\,Df(a)}$ the \emph{$(\eta,n-1)$ approximate Jacobian} of $f$ at $a$ and denote it by $(\eta,n-1)\ap\,Jf(a)$.

The following is a special case of Federer's coarea formula~\cite[Theorem 3.2.22]{federer}. Note that the original theorem works for any non-negative measurable function by the obvious application of the monotone convergence theorem.
\begin{theorem}[Federer's coarea formula] 
\label{coarea formula}
Suppose $W,Z\subset \R{n}$ are $\mathcal{H}^{n-1}$ measurable and $(\mathcal{H}^{n-1},n-1)$ rectifiable. If $f:W\rightarrow Z$ is Lipschitz, then for each $\mathcal{H}^{n-1}\mres W$ measurable non-negative function $g$ on $W$,
\begin{equation*}
\int_W g(x)\cdot (\mathcal{H}^{n-1}\mres W,n-1)\ap\,Jf(x)d\mathcal{H}^{n-1}(x)=\int_{Z}\int_{f^{-1}(z)}g(y)d\mathcal{H}^0(y)d\mathcal{H}^{n-1}(z).
\end{equation*}
\end{theorem}
It is implied in Theorem \ref{coarea formula} that $\int_{f^{-1}(z)}g(y)d\mathcal{H}^0(y)$ is $\mathcal{H}^{n-1}$ measurable as a function in $z$.
\section{Curvature Measures and $L_p$ Affine Surface Area}
\label{sec_new def}
In this section, we will prove the promised Theorem \ref{thm_new_2}, which will reveal the relationship between $L_p$ affine surface area and curvature measures. Notice that Theorem \ref{thm_new_1} follows by setting $p=1$ in Theorem \ref{thm_new_2} and the obvious fact that both sides of \eqref{eq_new_4} are translation invariant.

The following notations will be needed.

Let
\begin{align*}
T_1&=\left\{g:\partial K\rightarrow \mathbb{R}\,\,\mathcal{H}^{n-1}\text{ measurable}\colon 0<g<\infty\text{ and } \int_{\partial K}g^p(x)(h_K(\nu_K(x)))^{1-p}d\mathcal{H}^{n-1}(x)<\infty\right\},\\
T_2&=\left\{g:\partial K\rightarrow \mathbb{R} \text{ continuous}:g>0\right\}. 
\end{align*}

Note that when $K\in \mathcal{K}_0^n$, the sets $T_1$ and $T_2$ have the following relationship:
\begin{equation}
\label{star1}
T_1\supset T_2.
\end{equation}

Recall that $H^+$ is the set of normal boundary points with positive Gauss curvature (see \eqref{astr}).

\begin{proof}[Proof of Theorem \ref{thm_new_2}]
For the sake of simplicity, let us introduce the following notations:
\begin{align}
\label{eq_operator1}
L_1(g)&=\left(\int_{\partial K}g^{-n}(x)H_K(x)d\mathcal{H}^{n-1}(x)\right)^{\frac{p}{n+p}}\left(\int_{\partial K}g^{p}(x)(h_K(\nu_K(x)))^{1-p}d\mathcal{H}^{n-1}(x)\right)^{\frac{n}{n+p}},\\
\label{eq_operator2}
L_2(g)&=\left(\int_{\partial K}g^{-n}(x)dC_0(K,x)\right)^{\frac{p}{n+p}}\left(\int_{\partial K}g^p(x)(h_K(\nu_K(x)))^{1-p}dC_{n-1}(K,x)\right)^{\frac{n}{n+p}}.
\end{align}
Note that $L_1(g)$ is defined for $g\in T_1$, while $L_2(g)$ is defined for any positive Borel measurable function $g:\partial K\rightarrow \R{n}$ satisfying $\int_{\partial K}g^p(x)(h_K(\nu_K(x)))^{1-p}dC_{n-1}(K,x)<\infty$.

For each $g\in T_1$, by H\"{o}lder's inequality,
\begin{align}
\int_{\partial K}\left(\frac{H_K(x)}{(h_K(\nu_K(x)))^{(p-1)\frac{n}{p}}}\right)^{\frac{p}{n+p}}d\mathcal{H}^{n-1}(x)&=\int_{\partial K}\left(\frac{H_K(x)}{(h_K(\nu_K(x)))^{(p-1)\frac{n}{p}}}\right)^{\frac{p}{n+p}}g^{-\frac{np}{n+p}}(x)g^{\frac{np}{n+p}}(x)d\mathcal{H}^{n-1}(x) \nonumber \\ 
\label{eq_local 1003}
&\leq L_1(g).
\end{align}

For each $g\in T_2$, by \eqref{decom},
\begin{equation*}
\begin{aligned}
\int_{\partial K}g^{-n}(x)dC_0(K,x)&=\int_{\hat{\partial}K}g^{-n}(x)dC_0(K,x)+\int_{\partial K}g^{-n}(x)H_K(x)d\mathcal{H}^{n-1}(x)\\
&\geq \int_{\partial K}g^{-n}(x)H_K(x)d\mathcal{H}^{n-1}(x).
\end{aligned}
\end{equation*}
This, the fact that $\mathcal{H}^{n-1}(\beta)=C_{n-1}(K,\beta)$ for each Borel set $\beta\subset \partial K$, \eqref{eq_operator1}, and \eqref{eq_operator2} imply that,
\begin{equation}
\label{eq_local 1004}
L_1(g)\leq L_2(g),
\end{equation} 
for each $g\in T_2$.

Equations \eqref{eq_local 1003}, \eqref{eq_local 1004}, and the fact that $T_1\supset T_2$ show
\begin{equation}
\label{eq_local 1012}
\int_{\partial K}\left(\frac{H_K(x)}{(h_K(\nu_K(x)))^{(p-1)\frac{n}{p}}}\right)^{\frac{p}{n+p}}d\mathcal{H}^{n-1}(x)\leq \inf_{g\in T_1}L_1(g)\leq \inf_{g\in T_2}L_1(g)\leq \inf_{g\in T_2}L_2(g).
\end{equation}

To complete the proof, let us now show
\begin{equation}
\label{eq_local 1013}
\inf_{g\in T_2}L_2(g)\leq \int_{\partial K}\left(\frac{H_K(x)}{(h_K(\nu_K(x)))^{(p-1)\frac{n}{p}}}\right)^{\frac{p}{n+p}}d\mathcal{H}^{n-1}(x).
\end{equation}
It suffices, for every $\varepsilon>0$, to find a $g\in T_2$ such that
\begin{equation}
\label{eq_local 1011}
L_2(g)\leq \int_{\partial K}\left(\frac{H_K(x)}{(h_K(\nu_K(x)))^{(p-1)\frac{n}{p}}}\right)^{\frac{p}{n+p}}d\mathcal{H}^{n-1}(x)+\varepsilon.
\end{equation}

Let $\tilde{f}_i: \partial K\rightarrow (0,\infty)$ be defined by
\begin{equation}
\tilde{f}_i(x)=
\begin{cases}
i, & \text{if }x\in \hat{\partial}K,\\
h_K(\nu_K(x))^{\frac{p-1}{n+p}}H_K^{\frac{1}{n+p}}(x), &\text{if } x\in H^+\setminus \hat{\partial}K,\\
\frac{1}{i}, &\text{if }x\notin H^+\cup \hat{\partial}K.
\end{cases}
\end{equation}
Note that for each $\mathcal{H}^{n-1}$ measurable subset $A\subset \partial K$, there exists a Borel measurable set $\bar{A}\subset \partial K$ such that $\bar{A}\supset A$ and $\mathcal{H}^{n-1}(\bar{A}\setminus A)=0$.
This and the fact that $\hat{\partial}K$ is a Borel set ensure that we can modify the value of $\tilde{f}_i$ on a subset $Z\subset \partial K\setminus \hat{\partial} K$ with $\mathcal{H}^{n-1}(Z)=0$, such that the resulting function is a Borel measurable function. Denote the resulting function by $f_i$. Clearly $f_i(x)=\tilde{f}_i(x)=i$ for any $x\in \hat{\partial}K$ and $f_i(x)=\tilde{f}_i(x)$ for $\mathcal{H}^{n-1}$ almost all $x\in \partial K$. Define 
\begin{equation}
h_i(x)=
\begin{cases}
\frac{1}{i}, &\text{if }f_i(x)<\frac{1}{i},\\
f_i, &\text{if } \frac{1}{i}\leq f_i(x)\leq i,\\
i, &\text{if } i<f_i(x).
\end{cases}
\end{equation}
Note that both $f_i$ and $h_i$ are Borel measurable, and $\frac{1}{i}\leq h_i\leq i$.

By the fact that both $C_0(K,\cdot)$ and $C_{n-1}(K,\cdot)$ are finite measures, the assumption that $K$ contains the origin in its interior, \eqref{decom}, the fact that $\mathcal{H}^{n-1}(\beta)=C_{n-1}(K,\beta)$ for each Borel set $\beta\subset \partial K$, the choice of $f_i$, and $\mathcal{H}^{n-1}(\hat{\partial}K)=0$, we can compute the following limit,
\begin{align}
&\lim_{i\rightarrow \infty}\left(\int_{\partial K}(f_i^{-n}(x)+\frac{1}{i^n})dC_0(K,x)\right)^{\frac{p}{n+p}}\left(\int_{\partial K}(f_i^p(x)+\frac{1}{i^p})(h_K(\nu_K(x)))^{1-p}dC_{n-1}(K,x)\right)^{\frac{n}{n+p}}\nonumber \\
=&\lim_{i\rightarrow \infty}\left(\int_{\partial K}f_i^{-n}(x)dC_0(K,x)\right)^{\frac{p}{n+p}}\left(\int_{\partial K}f_i^p(x)(h_K(\nu_K(x)))^{1-p}dC_{n-1}(K,x)\right)^{\frac{n}{n+p}}\nonumber \\
=&\lim_{i\rightarrow \infty}\left(\int_{\partial K}f_i^{-n}(x)H_K(x)d\mathcal{H}^{n-1}(x)+\frac{1}{i^n}\int_{\hat{\partial}K}dC_0(K,x)\right)^{\frac{p}{n+p}}\nonumber\\
&\phantom{wewqewqewqewqewsdsdasdweqwewqewqew}\cdot\left(\int_{\partial K}f_i^p(x)(h_K(\nu_K(x)))^{1-p}dC_{n-1}(K,x)\right)^{\frac{n}{n+p}}\nonumber\\
=&\lim_{i\rightarrow \infty}\left(\int_{\partial K}f_i^{-n}(x)H_K(x)d\mathcal{H}^{n-1}(x)\right)^{\frac{p}{n+p}}\left(\int_{\partial K}f_i^p(x)(h_K(\nu_K(x)))^{1-p}d\mathcal{H}^{n-1}(x)\right)^{\frac{n}{n+p}}\nonumber\\
=&\lim_{i\rightarrow \infty}\left(\int_{H^+}f_i^{-n}(x)H_K(x)d\mathcal{H}^{n-1}(x)\right)^{\frac{p}{n+p}}\nonumber\\
&\phantom{weqwewqeq}\cdot\left(\int_{H^+}f_i^p(x)(h_K(\nu_K(x)))^{1-p}d\mathcal{H}^{n-1}(x)+\frac{1}{i^p}\int_{\partial K\setminus H^+}h_K(\nu_K(x))^{1-p}d\mathcal{H}^{n-1}(x)\right)^{\frac{n}{n+p}}\nonumber\\
=&\lim_{i\rightarrow \infty}\left(\int_{H^+}f_i^{-n}(x)H_K(x)d\mathcal{H}^{n-1}(x)\right)^{\frac{p}{n+p}}\left(\int_{H^+}f_i^p(x)(h_K(\nu_K(x)))^{1-p}d\mathcal{H}^{n-1}(x)\right)^{\frac{n}{n+p}}\nonumber\\
\label{eq_local 1008}
=& \int_{\partial K}\left(\frac{H_K(x)}{(h_K(\nu_K(x)))^{(p-1)\frac{n}{p}}}\right)^{\frac{p}{n+p}}d\mathcal{H}^{n-1}(x).
\end{align}

By \eqref{eq_operator2}, and the fact that $h_i^{-n}\leq f_i^{-n}+\frac{1}{i^n}$, $h_i^p\leq f_i^p+\frac{1}{i^p}$, 
\begin{equation}
\label{eq_local 1007}
\begin{aligned}
&\limsup_{i\rightarrow \infty}L_2(h_i)\\
\leq &\lim_{i\rightarrow \infty}\left(\int_{\partial K}(f_i^{-n}(x)+\frac{1}{i^n})dC_0(K,x)\right)^{\frac{p}{n+p}}\left(\int_{\partial K}(f_i^p(x)+\frac{1}{i^p})(h_K(\nu_K(x)))^{1-p}dC_{n-1}(K,x)\right)^{\frac{n}{n+p}}.
\end{aligned}
\end{equation}

Equations \eqref{eq_local 1008} and \eqref{eq_local 1007} imply that there exists $i_0$ such that 
\begin{equation}
\label{eq_local 1005}
L_2(h_{i_0})\leq \int_{\partial K}\left(\frac{H_K(x)}{(h_K(\nu_K(x)))^{(p-1)\frac{n}{p}}}\right)^{\frac{p}{n+p}}d\mathcal{H}^{n-1}(x)+\varepsilon/2.
\end{equation}

Note that $\frac{1}{i_0}\leq h_{i_0}\leq i_0$. Let $d\eta(x)=(h_K(\nu_K(x)))^{1-p}dC_{n-1}(K,x)+dC_0(K,x)$. Note that $\eta$ is a finite positive Borel measure on a compact metric space. Hence $\eta$ is regular. Since $\eta(\partial K)<\infty$, by Lusin's Theorem (see Theorem 2.24 in \cite{bigrudin}), there exists $\tilde{g}_j\in C(\partial K)$ such that 
\begin{equation}
\label{eq_local 10032}
\eta(\{x\in \partial K:\tilde{g}_j(x)\neq h_{i_0}(x)\})<\min\{\frac{1}{2i_0^n}, \frac{1}{2i_0^p}\}\frac{1}{j}.
\end{equation}

Let
\begin{equation*}
g_j(x)=
\begin{cases}
\frac{1}{i_0}, &\text{if } \tilde{g}_j(x)<\frac{1}{i_0},\\
\tilde{g}_j(x), &\text{if } \frac{1}{i_0}\leq \tilde{g}_j(x)\leq i_0,\\
i_0, &\text{if }i_0<\tilde{g}_j.
\end{cases}
\end{equation*} 
It is easy to see that $g_j$ is still continuous and $\frac{1}{i_0}\leq g_j\leq i_0$. Moreover, since whenever $g_j(x)\neq h_{i_0}(x)$, it must be the case that $\tilde{g}_j(x)\neq h_{i_0}(x)$, we have by \eqref{eq_local 10032},
\begin{equation*}
\eta(\{x\in \partial K:g_j\neq h_{i_0}\})<\min\{\frac{1}{2i_0^n}, \frac{1}{2i_0^p}\}\frac{1}{j}.
\end{equation*}
Hence,
\begin{equation*}
\begin{aligned}
&\left|\int_{\partial K}h_{i_0}^{-n}(x)dC_0(K,x)-\int_{\partial K}g_j^{-n}(x)dC_0(K,x)\right|\leq \frac{1}{2i_0^n}\frac{1}{j}2i_0^n=\frac{1}{j},\\
&\left|\int_{\partial K}h_{i_0}^p(x)(h_K(\nu_K(x)))^{1-p}dC_{n-1}(K,x)-\int_{\partial K}g_j^p(x)(h_K(\nu_K(x)))^{1-p}dC_{n-1}(K,x)\right|\leq \frac{1}{2i_0^p}\frac{1}{j}2i_0^p=\frac{1}{j}.
\end{aligned}
\end{equation*}
This implies that $\lim_{j\rightarrow \infty}L_2(g_j)=L_2(h_{i_0})$. As a result, there exists $j_0$ such that 
\begin{equation}
\label{eq_local 1010}
L_2(g_{j_0})\leq L_2(h_{i_0})+\varepsilon/2.
\end{equation}

Choose $g=g_{j_0}$. By \eqref{eq_local 1005} and \eqref{eq_local 1010}, such a $g$ will satisfy \eqref{eq_local 1011}.
\end{proof}

It is immediate from \eqref{eq_local 1012} and \eqref{eq_local 1013} that
\begin{equation}
\label{eq_local 1014}
\begin{aligned}
&\inf_{g\in T_1}\left\{\left(\int_{\partial K}g^{-n}(x)H_K(x)d\mathcal{H}^{n-1}(x)\right)^{\frac{p}{n+p}}\left(\int_{\partial K}g^p(x)(h_K(\nu_K(x)))^{1-p}d\mathcal{H}^{n-1}(x)\right)^{\frac{n}{n+p}}\right\}\\
=&\inf_{g\in T_2}\left\{\left(\int_{\partial K}g^{-n}(x)dC_0(K,x)\right)^{\frac{p}{n+p}}\left(\int_{\partial K} g^p(x)(h_K(\nu_K(x)))^{1-p}dC_{n-1}(K,x)\right)^{\frac{n}{n+p}}\right\}.
\end{aligned}
\end{equation}

Theorem \ref{thm_new_2} suggests, in addition to \eqref{eq_new_6}, \eqref{eq_new_7}, and \eqref{eq_new_8}, there is a new representation of the existing notion of $L_p$ affine surface area:

\begin{definition}Let $p>0$ be a real number and $K\subset \R{n}$ be a convex body that contains the origin in its interior. The $L_p$ affine surface area of $K$ can be defined by
\begin{equation}
\label{eq_new_9}
\Omega_p(K)=\inf_{g}\left\{\left(\int_{\partial K}g^{-n}(x)dC_0(K,x)\right)^{\frac{p}{n+p}}\left(\int_{\partial K} g^p(x)(h_K(\nu_K(x)))^{1-p}dC_{n-1}(K,x)\right)^{\frac{n}{n+p}}\right\},
\end{equation} 
where the infimum is taken over all positive continuous functions $g:\partial K\rightarrow \mathbb{R}$.
\end{definition}

Note that \eqref{eq_new_9} is the $L_p$ analogue of \eqref{eq_new_5}.

It is worthwhile to point out that, as can be seen from the proofs in this section, Definition \eqref{eq_new_9} is ``dual'' to Definition \eqref{eq_new_8}. Since it has already been established that Definitions \eqref{eq_new_6},  \eqref{eq_new_7}, \eqref{eq_new_8} are the same, it follows from Theorem \ref{thm_new_2} that Definition \eqref{eq_new_9} is equivalent to Definition \eqref{eq_new_6}. However, we wish to give a direct proof of the equivalence between Definitions \eqref{eq_new_9} and \eqref{eq_new_6} as this will further reveal the relationship between the two formulations of $L_p$ affine surface area. This will be carried out in Section \ref{sec_equivalences}.

\section{Properties of $L_p$ Affine Surface Area}
\label{sec_properties}
In this section, some basic properties of $L_p$ affine surface area will be shown using the new representation \eqref{eq_new_9} (\eqref{eq_new_5} if $p=1$). Although the properties given in this section are not new, different proofs are given using Definition \eqref{eq_new_9} that do not depend on any results about $L_p$ affine surface area established using the previously known forms.

The following proposition shows that $\Omega_p$ is a homogeneous functional on the set of convex bodies containing the origin in their interiors.

\begin{prop}
\label{prop_tmp 100}
Let $p>0$ be a real number. Suppose $K\in \mathcal{K}_0^n$. For $\lambda> 0$, $\Omega_p(\lambda K)=\lambda^{\frac{n(n-p)}{n+p}}\Omega_p(K)$.
\end{prop}
\begin{proof}
For each positive continuous $g$ on $\partial K$, define $\tilde{g}$ on $\partial (\lambda K)$ by
\begin{equation}
\label{eq_new_10}
\tilde{g}(y)=g(x),\phantom{asdweqewe} \text{if } y=\lambda x.
\end{equation}
Notice that for each Borel set $\beta\subset\partial K$,
\begin{equation}
\label{eq_new_11}
C_0(\lambda K, \lambda\beta)=\mathcal{H}^{n-1}(\nu({\lambda K},\lambda\beta))=\mathcal{H}^{n-1}(\nu(K,\beta))=C_0(K,\beta).
\end{equation}
By \eqref{eq_new_11} and \eqref{eq_new_10},
\begin{equation}
\label{eq_new_12}
\begin{aligned}
\int_{\partial (\lambda K)}\tilde{g}^{-n}(y)dC_0(\lambda K,y)=&\int_{\partial K}\tilde{g}^{-n}(\lambda x)dC_0(K,x)\\
=&\int_{\partial K}g^{-n}(x)dC_0(K,x).
\end{aligned}
\end{equation}
By the homogeneity of the support function,
\begin{equation}
\label{eq_new_13}
\begin{aligned}
\int_{\partial (\lambda K)}\tilde{g}^p(y)(h_{\lambda K}(\nu_{\lambda K}(y)))^{1-p}dC_{n-1}(\lambda K,y)
=& \int_{\partial K}\tilde{g}^{p}(\lambda x)\lambda^{1-p}(h_K(\nu_K(x)))^{1-p}\lambda^{n-1}dC_{n-1}(K,x)\\
=&\lambda^{n-p}\int_{\partial K}g^p(x)(h_K(\nu_K(x)))^{1-p}dC_{n-1}(K,x).
\end{aligned}
\end{equation}
By \eqref{eq_new_9}, \eqref{eq_new_12}, and \eqref{eq_new_13}, we get the desired result.
\end{proof}

It is trivial to see that when $p=1$, translation invariance is satisfied by $\Omega_1$.
\begin{prop}
Let $p>0$ be a real number. Suppose $K$ is polytope that contains the origin in its interior. Then,
\begin{equation*}
\Omega_p(K)=0.
\end{equation*}
\end{prop}
\begin{proof}
Note that the measure $C_0(K,\cdot)$ in this case is concentrated on the set of vertices of $K$, which has only finitely many points. This implies that we can let the second integral in \eqref{eq_new_9} be arbitrarily small while holding the value of the first integral constant. Hence $\Omega_p(K)=0$.
\end{proof}

Another important property of $L_p$ affine surface area that historically took a long time to prove (settled in \cite{lutwak1991,lutwak1996}) is its upper semi-continuity.

\begin{theorem}
{
\label{prop_new_1}
Let $p>0$ be a real number. Then, $\Omega_p$ is upper semi-continuous with respect to the Hausdorff metric.
}
\end{theorem}
\begin{proof}
This is a direct result from the weak continuity of $C_0(K,\cdot)$, $C_{n-1}(K,\cdot)$ and the fact that the infimum of continuous functionals is upper semi-continuous.
\end{proof}

By \eqref{eq_local 1014}, the following variant of \eqref{eq_new_9} will give us more flexibility in choosing the function $g$:
\begin{equation}
\label{eq_new_14}
\Omega_p(K)=\inf_{g\in T_1}\left\{\left(\int_{\partial K}g^{-n}(x)H_K(x)d\mathcal{H}^{n-1}(x)\right)^{\frac{p}{n+p}}\left(\int_{\partial K}g^p(x)(h_K(\nu_K(x)))^{1-p}d\mathcal{H}^{n-1}(x)\right)^{\frac{n}{n+p}}\right\}.
\end{equation}

We will need the following two lemmas, which were established by Sch\"{u}tt \& Werner in \cite{MR2074173}. 

\begin{lemma}
Let $K\in \mathcal{K}_0^n$ and $\phi:\R{n}\rightarrow \R{n}$ be a linear map with determinant either $1$ or $-1$. For each integrable function $f:\partial K\rightarrow \mathbb{R}$,
\begin{equation}
\label{tmp501}
\int_{\partial K}f(x)d\mathcal{H}^{n-1}(x)=\int_{\partial(\phi K)}f(\phi^{-1}(y))\vectornorm{\phi^{-t}(\nu_K(\phi^{-1}(y)))}^{-1}d\mathcal{H}^{n-1}(y).
\end{equation}  
\end{lemma}

\begin{lemma}
Let $K\in \mathcal{K}_0^n$ and $\phi:\R{n}\rightarrow \R{n}$ be a linear map with determinant either $1$ or $-1$. Suppose $x$ is a normal boundary point of $K$. Then $\phi(x)$ is a normal boundary point of $\phi K$ and moreover,
\begin{equation}
\label{tmp502}
H_K(x)=\vectornorm{\phi^{-t}(\nu_K(x))}^{n+1}H_{\phi K}(\phi(x)).
\end{equation}
\end{lemma}

The next proposition shows that $\Omega_p$ is invariant under volume preserving linear transformations.
\begin{prop}
\label{prop_tmp 101}
Let $p>0$ be a real number. Suppose $K\in \mathcal{K}_0^n$ and $\phi:\R{n}\rightarrow\R{n}$ is a linear map with determinant either $1$ or $-1$. Then
\begin{equation*}
\Omega_p(\phi K)=\Omega_p(K).
\end{equation*}
\end{prop}
\begin{proof}
Note that $|\det(\phi^{-1})|=1$ and $\phi^{-1}(\phi K)=K$. Thus, we only need to show $\Omega_p(K)\geq \Omega_p(\phi K)$.

If $x$ is a regular boundary point of $K$, then $\phi (x)$ is a regular boundary point of $\phi K$ and 
\begin{equation}
\label{tmp503}
\nu_{\phi K}(\phi (x))=\frac{\phi^{-t}(\nu_K(x))}{\vectornorm{\phi^{-t}(\nu_K(x))}}.
\end{equation}
Hence,
\begin{equation}
\label{tmp504}
\vectornorm{\phi^t (\nu_{\phi K}(\phi(x)))}=\vectornorm{\phi^{-t}(\nu_K(x))}^{-1}.
\end{equation}
The definitions of support function and outer unit normal, together with \eqref{tmp503} and \eqref{tmp504}, imply
\begin{equation}
\label{tmp505}
h_K(\nu_K(x))=\vectornorm{\phi^t(\nu_{\phi K}(\phi(x)))}^{-1}h_{\phi K}(\nu_{\phi K}(\phi(x))).
\end{equation}

Let $g\in T_1$ be a function defined on $\partial K$ such that $\int_{\partial K}g^{-n}(x)H_K(x)d\mathcal{H}^{n-1}(x)<\infty$. By \eqref{tmp501}, \eqref{tmp502}, and \eqref{tmp504}, 
\begin{equation}
\label{tmp506}
\int_{\partial K}g^{-n}(x)H_K(x)d\mathcal{H}^{n-1}(x)=\int_{\partial(\phi K)}\vectornorm{\phi^t(\nu_{\phi K}(y))}^{-n}g^{-n}(\phi^{-1}(y))H_{\phi K}(y)d\mathcal{H}^{n-1}(y).
\end{equation}
By \eqref{tmp501}, \eqref{tmp504}, and \eqref{tmp505}, 
\begin{equation}
\label{tmp507}
\int_{\partial K}g^p(x)(h_K\circ \nu_K(x))^{1-p}d\mathcal{H}^{n-1}(x)=\int_{\partial (\phi K)}\vectornorm{\phi^{t}(\nu_{\phi K}(y))}^p g^p(\phi^{-1}(y))(h_{\phi K}(\nu_{\phi K}(y)))^{1-p}d\mathcal{H}^{n-1}(y).
\end{equation}
Let $\tilde{g}:\partial(\phi K)\rightarrow \mathbb{R}$ be defined as 
\begin{equation*}
\tilde{g}(x)=
\begin{cases}
g(\phi^{-1}(y))\vectornorm{\phi^t (\nu_{\phi K}(y))}, &\text{if }y\text{ is a regular boundary point of }K,\\
1,&\text{otherwise}.
\end{cases}
\end{equation*}
The fact that $\mathcal{H}^{n-1}$ almost all points on the boundary of a convex body are regular, the choice of $g$, and \eqref{tmp507} show that $\tilde{g}$ is a positive, $\mathcal{H}^{n-1}$ measurable function on $\partial(\phi K)$ and $$\int_{\partial(\phi K)}\tilde{g}^{p}(y)(h_{\phi K}(\nu_{\phi K}(y)))^{1-p}d\mathcal{H}^{n-1}(y)<\infty.$$ 

By \eqref{tmp506}, \eqref{tmp507}, and \eqref{eq_new_14}, we immediately have $\Omega_p(K)\geq\Omega_p(\phi K)$.
\end{proof}

An immediate corollary of Proposition \ref{prop_tmp 100} and Proposition \ref{prop_tmp 101} is:
\begin{cor}
Let $p>0$ be a real number. Suppose $K\in \mathcal{K}_0^n$ and $\phi:\R{n}\rightarrow \R{n}$ is an invertible linear map. Then
\begin{equation*}
\Omega_p(\phi K)=|\det(\phi)|^{\frac{n-p}{n+p}}\Omega_p(K).
\end{equation*}
\end{cor}

The $L_p$ affine surface area functional is also a valuation. That is, if $K,L\in \mathcal{K}_0^n$ are such that $K\cup L\in \mathcal{K}_0^n$, then
\begin{equation*}
\Omega_p(K\cap L)+\Omega_p(K\cup L)=\Omega_p(K)+\Omega_p(L).
\end{equation*}
This property, however, is not immediate under the new form \eqref{eq_new_9}. The reader is recommended to see e.g., Sch\"{u}tt \cite{schutt1993}, Ludwig \& Reitzner \cite{ludwig,MR2680490} (and the references therein) for a proof of the valuation property of $L_p$ affine surface area and the role of $L_p$ affine surface area in the theory of valuation.

For each $g\in T_1$, denote
\begin{equation*}
\begin{aligned}
V_p(K,g)&=\frac{1}{n}\int_{\partial K}g^p(x)(h_K(\nu_K(x)))^{1-p}d\mathcal{H}^{n-1}(x),\\
W(K,g)&=\frac{1}{n}\int_{\partial K}g^{-n}(x)H_K(x)d\mathcal{H}^{n-1}(x).
\end{aligned}
\end{equation*}
Notice that for a convex body $L$ that contains the origin in its interior,
\begin{equation*}
V_p(K,h_L\circ\nu_K)=V_p(K,L).
\end{equation*}
Here $V_p(K,L)$ is the $L_p$ mixed volume of $K$ and $L$ and can be defined by
\begin{equation*}
V_p(K,L)=\frac{1}{n}\int_{S^{n-1}}h_L^p(u)h_K^{1-p}(u)dS_K(u).
\end{equation*}

Built in \eqref{eq_new_14} is the following $L_p$ mixed volume type inequality,
\begin{equation}
\label{eq_new_15}
\frac{1}{n}\Omega_p(K)\leq W(K,h_L\circ\nu_K)^{\frac{p}{n+p}}V_p(K,L)^{\frac{n}{n+p}}.
\end{equation}

We will postpone the proof of the $L_p$ affine isoperimetric inequality (the $L_p$ analogue of the celebrated affine isoperimetric inequality \eqref{star}, which was given by Lutwak in \cite{lutwak1996} for $p\geq1$, and by Werner \& Ye in \cite{MR2414321} for all other $p$) to Section \ref{sec_equivalences}. The proof will utilize \eqref{eq_new_15}.

\section{Lipschitz Property of Restrictions of $\tau_K$}
\label{sec_lip}
Recall that $\nu_K$ is the Gauss map defined on $\reg K$, the set of regular boundary points of $K$, and $\tau_K$ is the inverse Gauss map defined on $\regn K$, the set of regular normal vectors of $K$. See Section \ref{section_preli}.

One of the essential difficulties people encounter when trying to prove the equivalence between Definitions \eqref{eq_new_7} and \eqref{eq_new_8}, as well as Definitions \eqref{eq_new_6} and \eqref{eq_new_9}, is linking an integral over the unit sphere (image of the Gauss map or domain of the inverse Gauss map) with an integral over the boundary of a convex body (domain of the Gauss map or image of the inverse Gauss map). One of the direct bridges was built in \cite{daniel1996} by exploring the Lipschitz property of restrictions of $\nu_K$. 

In particular, for each $r>0$ and each convex body $K\subset \R{n}$, denote 
\begin{equation*}
(\partial K)_r=\left\{x\in \partial K: \exists a\in S^{n-1} \text{ such that } B(x-ra,r)\subset K\right\}.
\end{equation*}

The following lemma was shown in \cite{MR0367810}:

\begin{lemma}
\label{Lemma 4.1}
For each convex body $K\subset \R{n}$, $$\mathcal{H}^{n-1}(\partial K\setminus(\cup_{r>0}(\partial K)_r)=0.$$
\end{lemma}

Denote $\nu_K|_{(\partial K)_r}$ by $\nu_r$. Hug in \cite[Lemmas 2.1, 2.3]{daniel1996} showed that $\nu_r$ is a Lipschitz map and calculated the approximate Jacobian of $\nu_r$:

\begin{lemma}
\label{Lemma 4.2}
Let $K\subset \R{n}$ be a convex body and $r>0$. The following results are true:
\begin{enumerate}
\item $(\partial K)_r$ is a closed subset of $\partial K$, and $\nu_r$ is Lipschitz;
\item For $\mathcal{H}^{n-1}$ almost all $x\in (\partial K)_r$, we have 
\begin{equation*}
(\mathcal{H}^{n-1}\mres {(\partial K)_r},n-1)\,\ap\,J\nu_r(x)=H_K(x).
\end{equation*}
\end{enumerate}
\end{lemma}

The following characterization of points at which the generalized curvature function is positive was established in \cite[Lemma 2.7]{daniel1996}:

\begin{lemma}
\label{Lemma 4.3}
Let $K\subset\R{n}$ be a convex body. Suppose $u_0\in D^2(h_K)\cap S^{n-1}$. Then the following two conditions are equivalent:
\begin{enumerate}
\item There is some $r>0$ such that $B(\tau_K(u_0)-ru_0,r)\subset K$.
\item $F_K(u_0)>0$.
\end{enumerate}
\end{lemma}

In this section, the Lipschitz property of restrictions of $\tau_K$ will be discussed, which will be useful in proving the equivalence between Definitions \eqref{eq_new_6} and \eqref{eq_new_9}. In particular, we will divide the unit sphere into a countable collection of subsets (up to a set of measure 0), such that $\tau_K$ restricted to each subset is Lipschitz. Thus a change of variable will be made possible by using Federer's coarea formula.

Let $K\subset \R{n}$ be a convex body. For each $i\in \mathbb{N}^+$, we define $A_i\subset S^{n-1}$ by
\begin{equation*}
A_i=\{u\in S^{n-1}:\exists x\in \tau(K,\{u\})\text{ such that }K\subset B[x-iu,i]\}.
\end{equation*}
Here $\tau(K,\{u\})$ is the inverse spherical image of $\{u\}$ (see Section \ref{section_preli}).
\begin{remark}
Note that it is easily seen that if $u\in A_i$, then $u$ must be a regular normal vector of $K$. Hence,
\begin{equation}
A_i=\{u\in \regn K:K\subset B[x-iu,i], \text{where }x=\tau_K(u)\}.
\end{equation}
\end{remark}

We will denote the restriction of $\tau_K$ to $A_i$ by $\tau_i$, for each $i\in \mathbb{N}^+$.

The following lemma was observed by Hug in \cite{daniel1996}, which can be proved similarly to Lemma 2.7 in \cite{daniel1996}.
\begin{lemma}
\label{lemma_a.e.}
Let $K\subset \R{n}$ be a convex body. For each $u_0\in D^2(h_K)\cap S^{n-1}$, there exists $i\in \mathbb{N}^+$ such that
\begin{equation}
\label{tmp2}
K\subset B[x_0-iu_0,i],
\end{equation}
where $x_0=\tau_K(u_0)$. 
\end{lemma}

The following corollary follows immediately from Lemma \ref{lemma_a.e.} and \eqref{hkdifferentiable}.
\begin{cor}
\label{corollary_a.e.}
Let $K\subset \R{n}$ be a convex body. With respect to $\mathcal{H}^{n-1}$, almost every point in $S^{n-1}$ is contained in $\cup_{i=1}^{\infty}A_i$, i.e., 
\begin{equation*}
\mathcal{H}^{n-1}(S^{n-1}\setminus \cup_{i=1}^{\infty}A_i)=0.
\end{equation*}
\end{cor}

The following lemma will be needed.

\begin{lemma}
\label{lemma_closedness}
Let $K\subset \R{n}$ be a convex body. For each $i\in \mathbb{N}^+$, $A_i$ is closed.
\end{lemma} 
\begin{proof}
Suppose $\{u_j\}_{j=1}^{\infty}$ is a convergent sequence in $A_i$. Denote by $u_0$ its limit. Let $x_j=\tau_K(u_j)$. Since $\partial K$ is compact, we can take a convergent subsequence of $\{x_j\}_{j=1}^{\infty}$. We denote the subsequence and the corresponding subsequence of $\{u_j\}_{j=1}^{\infty}$ again by $\{x_j\}_{j=1}^{\infty}$ and $\{u_j\}_{j=1}^{\infty}$. Denote by $x_0$ the limit of $\{x_j\}_{j=1}^{\infty}$. Let $P_0$ be the hyperplane that passes $x_0$ and has $u_0$ as its normal. Since $x_j=\tau_K(u_j)$,
\begin{equation*}
\langle x,u_j\rangle\leq \langle x_j,u_j\rangle,
\end{equation*}
for each $x\in K$. Let $j\rightarrow \infty$, we have 
\begin{equation*}
\langle x, u_0\rangle\leq \langle x_0,u_0\rangle,
\end{equation*}
for each $x\in K$. Hence $P_0$ is a supporting hyperplane of $K$ at $x_0$ and $x_0\in \tau(K,\{u_0\})$. Since $\{u_j\}_{j=1}^{\infty}\subset A_i$, we have $K\subset B[x_j-iu_j,i]$. Let $j\rightarrow \infty$. We have $K\subset B[x_0-iu_0,i]$, where $x_0\in \tau(K,\{u_0\})$. This implies that $u_0\in A_i$. Hence $A_i$ is closed.
\end{proof}

\begin{lemma}
\label{Lemma_tau lip}
For each $i\in \mathbb{N}^+$ and each convex body $K\subset \R{n}$, $\tau_i$ is Lipschitz.
\end{lemma}
\begin{proof}
Let $u_1,u_2\in A_i$. Denote by $\theta\in [0,\pi]$ the angle formed by $u_1$ and $u_2$. Let $x_1=\tau_i(u_1)$, and $x_2=\tau_i(u_2)$. 

We first assume $0<\theta<\pi/2$. Suppose $x_1\neq x_2$. Otherwise, there is nothing to prove. Since $u_1,u_2$ are not parallel to each other, the points $x_1-iu_1$ and $x_2-iu_2$ cannot both lie on the line passing $x_1,x_2$. Suppose $x_1-iu_1$ does not lie on the line passing $x_1,x_2$. Denote $x_1-iu_1$ by $C$. Without loss of generality, since this lemma is invariant under translation, we may assume that $C$ is the origin. Let $P$ be the two dimensional subspace spanned by $x_1$ and $x_2$. Note that $u_1$ is parallel to $x_1$. Write 
\begin{equation}
\label{100}
u_2=v_2+w_2,
\end{equation}
where $v_2\in P$ and $w_2\in P^\perp$. Since $u_1$ and $u_2$ are not perpendicular, we have $v_2\neq 0$. Let $\tilde{u_2}=\frac{v_2}{\vectornorm{v_2}}\in P$. Denote by $\tilde{\theta}\in[0,\pi]$ the angle formed by $u_1$ and $\tilde{u_2}$. Notice that by the definition of $\tilde{u_2}$, \eqref{100}, and the fact that $u_1\in P$, 
\begin{equation}
\label{tmp100}
\cos \tilde{\theta}=\langle\tilde{u_2},u_1\rangle=\frac{1}{\vectornorm{v_2}}\langle v_2, u_1\rangle=\frac{1}{\vectornorm{v_2}}\langle u_2, u_1\rangle=\frac{1}{\vectornorm{v_2}}\cos \theta.
\end{equation}
This implies that if $0<\theta<\pi/2$, we have $0\leq \tilde{\theta}<\pi/2$. In this case, by \eqref{tmp100} and that $\vectornorm{v_2}\leq 1$, we have $\cos\tilde{\theta}\geq \cos\theta$. By the monotonicity of the cosine function on $[0,\pi/2)$, we have 
\begin{equation}
\label{tmp300}
\tilde{\theta}\leq \theta,
\end{equation}
when $0< \theta<\pi/2$.

By the definition of $A_i$, the fact that $u_1\in A_i$ implies that $P\cap K$ is a subset of the disc $P\cap B[x_1-iu_1,i]$. Note that $P\cap K$ is non-empty and is either the line segment joining $x_1$ and $x_2$ or a convex body in $P$. In either case, for any $x\in P\cap K$, by the definition of $\tilde{u_2}$, \eqref{100}, the fact that $x_2=\tau_i(u_2)$, \eqref{100}, and the definition of $\tilde{u_2}$ once again,
\begin{equation}
\label{tmp200}
\langle \tilde{u_2},x\rangle=\frac{1}{\vectornorm{v_2}}\langle v_2,x\rangle=\frac{1}{\vectornorm{v_2}}\langle u_2,x\rangle \leq \frac{1}{\vectornorm{v_2}}\langle u_2,x_2\rangle=\frac{1}{\vectornorm{v_2}}\langle v_2,x_2\rangle=\langle\tilde{u_2},x_2\rangle.
\end{equation}

Now, we show that $\tilde{\theta}\neq 0$ as a result of $x_1\neq x_2$. If otherwise, $\tilde{u_2}=u_1$. Since $P\cap K\subset P\cap B[x_1-iu_1,i]$, the line $l_{x_1}\subset P$ passing $x_1$ and perpendicular to $u_1$ can only intersect $P\cap K$ at $x_1$. Note also that since $x_1=\tau_i(u_1)$, we have $\langle u_1,x\rangle \leq \langle u_1,x_1\rangle$ for each $x\in P\cap K$. But \eqref{tmp200} implies that $\langle u_1,x_1\rangle =\langle \tilde{u_2},x_1\rangle \leq \langle \tilde{u_2}, x_2\rangle =\langle u_1,x_2\rangle$. Hence $\langle u_1, x_2\rangle =\langle u_1,x_1\rangle $ and as a result, we have $x_2\in l_{x_1}$. This immediately implies $x_1=x_2$, which contradicts with our assumption. Thus, it suffices to look at the case when $0<\tilde{\theta}\leq \theta<\pi/2$, with $x_1\neq x_2$. In this case, let $l\subset P$ be the line passing through $x_2$ and is perpendicular to $\tilde{u_2}$. Extend $Cx_1$ so that it intersects $l$ at $D$. Starting from $C$, make an ray in the direction of $\tilde{u_2}$ (perpendicular to $l$ as a result), so that it crosses the boundary of $B[x_1-iu_1,i]$ at $E$. Let $\phi,\psi$ be the angles indicated in Figure \ref{fig:1}. By \eqref{tmp200} and that $0<\tilde{\theta}<\pi/2$, the point $D$ does not belong to the interior of $B(x_1-iu_1,i)$. Hence $\phi\geq \psi=\pi/2-\tilde{\theta}$. 
\begin{figure}{
\centering
{
 \includegraphics{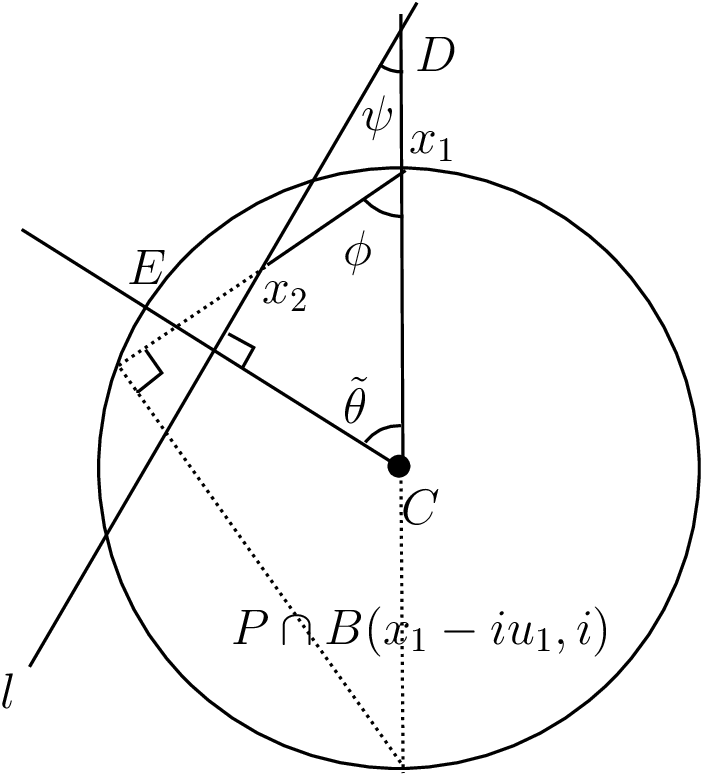}
\caption{}
\label{fig:1}}       
}\end{figure}
The fact that $x_2\in B[x_1-iu_1,i]$, together with \eqref{tmp300}, implies
\begin{equation}
\vectornorm{x_1-x_2}\leq 2i\cos \phi\leq 2i \cos (\pi/2-\tilde{\theta})=2i\sin\tilde{\theta}\leq 2i\sin \theta.
\end{equation}
Observe that $\vectornorm{u_1-u_2}=2\sin \frac{\theta}{2}$. Since 
\begin{equation*}
\lim_{\theta \rightarrow 0}\frac{2i\sin \theta}{2\sin \frac{\theta}{2}}=2i,
\end{equation*}
we conclude that there exists $0<\delta_0<\pi/2$, such that
\begin{equation}
\label{tmp400}
\vectornorm{x_1-x_2}\leq 3i\vectornorm{u_1-u_2},
\end{equation}
for any $u_1,u_2\in A_i$, satisfying $0<\theta<\delta_0$. (Note that \eqref{tmp400} trivially holds if $x_1=x_2$.)

For the case $\delta_0\leq \theta\leq\pi$, since $\vectornorm{u_1-u_2}=2\sin\frac{\theta}{2}$ and $K$ is bounded, we have
\begin{equation*}
\frac{\vectornorm{x_1-x_2}}{\vectornorm{u_1-u_2}}
\end{equation*}
is bounded from above. This and \eqref{tmp400} prove the existence of $M>0$ such that
\begin{equation*}
\vectornorm{x_1-x_2}\leq M\vectornorm{u_1-u_2},
\end{equation*}
for any $u_1,u_2\in A_i$. Hence $\tau_i$ is Lipschitz.
\end{proof}

The following characterization of normal boundary points with positive curvature was shown in~\cite{daniel1996} as Corollary 3.2 (see also \cite{leichtweiss1998}).

\begin{lemma}
\label{lemma_H>0}
Let $K\subset\R{n}$ be a convex body and $x_0$ be a normal boundary point. The following two conditions are equivalent:
\begin{enumerate}
\item \label{enumerate_1}$H_K(x_0)>0$;
\item \label{enumerate_2}there exists $i \in \mathbb{N}^+$ such that $K\subset B[x_0-iu_0,i]$ where $u_0=\nu_K(x_0)$.
\end{enumerate}
\end{lemma}

The following lemma was proved in~\cite[Lemma 2.5]{daniel1996}.

\begin{lemma}
\label{Lemma_reciporical of H and F}
Let $K\subset \R{n}$ be a convex body. Suppose $x_0$ is a normal boundary point and $u_0=\nu_K(x_0)\in D^2(h_K)$. Then $H_K(x_0)F_K(u_0)=1$.
\end{lemma}
There is still one piece missing that hinders us from applying Federer's coarea formula to $\tau_i$, namely, the Jacobian of $\tau_i$. Taking the Jacobian of $\tau_i$ in the classical sense is impossible since the classical Jacobian requires $\tau_i$ to be defined in an open set. Therefore we have to consider the approximate Jacobian of $\tau_i$ instead. 

\begin{lemma}
\label{Lemma_Jacobian}
Let $K\subset \R{n}$ be a convex body. For each given $i\in \mathbb{N}^+$,
\begin{equation}
(\mathcal{H}^{n-1}\mres {A_i},n-1)\,ap\,J\tau_i(u)=F_K(u),
\end{equation}
for $\mathcal{H}^{n-1}$ almost all $u\in A_i$.
\end{lemma}
\begin{proof}
Observe that $A_i$ is an $(\mathcal{H}^{n-1},n-1)$ rectifiable and $\mathcal{H}^{n-1}$ measurable subset (since $A_i$ is closed by Lemma \ref{lemma_closedness}) of $\mathbb{R}^n$. By Theorem 3.2.19 in \cite{federer}, for $\mathcal{H}^{n-1}$ almost all $u_0\in A_i$, $\Tan^{n-1}(\mathcal{H}^{n-1}\mres A_i,u_0)$ is an $(n-1)$ dimensional subspace of $\R{n}$. By definition, $$\Tan^{n-1}(\mathcal{H}^{n-1}\mres A_i,u_0)\subset \Tan(S^{n-1},u_0)=u_0^\perp.$$ Hence, 
\begin{equation}
\label{eq_tangent plane calc}
\Tan^{n-1}(\mathcal{H}^{n-1}\mres A_i,u_0)= u_0^\perp,
\end{equation}
for $\mathcal{H}^{n-1}$ almost all $u_0\in A_i$. This and \eqref{hkdifferentiable} imply that $\mathcal{H}^{n-1}$ almost all vectors in $A_i$ are in $D^2(h_K)$ and satisfy \eqref{eq_tangent plane calc}. Thus, to prove this lemma, we may assume that $u_0\in D^2(h_K)$ and $u_0$ satisfies \eqref{eq_tangent plane calc}.

Let $\varepsilon>0$ be a real number and $\vartheta:\R{n}\rightarrow \R{n}$ be an arbitrary subgradient choice of $h_K$. The fact that $u_0\in D^2(h_K)$ implies that there exists $\delta_0>0$ and a symmetric linear map $Ah_K(u_0):\R{n}\rightarrow \R{n}$ such that 
\begin{equation*}
\{u\in A_i:\vectornorm{\vartheta(u)-\vartheta(u_0)-A{h_K}(u_0)(u-u_0)}>\varepsilon\vectornorm{u-u_0}\}\cap B(u_0,\delta)=\emptyset,
\end{equation*}
for $0<\delta\leq \delta_0$.

Since each $u\in A_i$ is a regular normal vector, we have $\vartheta(u)=\tau_i(u)$ for each $u\in A_i$. Hence by the definition of $\Theta^{n-1}$, we have
\begin{align*}
\Theta^{n-1}\left(\left\{(\mathcal{H}^{n-1}\mres A_i)\mres {\left(\R{n}\backslash
\{u\in A_i:\vectornorm{\tau_i(u)-\tau_i(u_0)-Ah_K(u_0)(u-u_0)}\right.}\right.\right.\\ \leq\left.\left.{\left.\varepsilon\vectornorm{u-u_0}\}\right)},u_0\right\}\right)=0.
\end{align*}

By the definition of $(\mathcal{H}^{n-1}\mres {A_i},n-1)\,\ap\,D\tau_i(u_0)$ and \eqref{eq_tangent plane calc}, we have
\begin{equation*}
(\mathcal{H}^{n-1}\mres A_i,n-1)\ap\,D\tau_i(u_0)=\left.Ah_K(u_0)\right|_{\Tan^{n-1}({\mathcal{H}^{n-1}\mres A_i,u_0)}}=\left.Ah_K(u_0)\right|_{u_0^\perp}.
\end{equation*}
By this, the definition of $(\mathcal{H}^{n-1}\mres A_i,n-1)\ap\,J\tau_i(u_0)$, that $u_0^\perp$ is an invariant subspace of $Ah_K(u_0)$, and the definition of the generalized curvature function, we have 
\begin{equation*}
(\mathcal{H}^{n-1}\mres {A_i},n-1)\,\ap\,J\tau_i(u_0)=F_K(u_0).
\end{equation*}
\end{proof}
\section{Equivalence between Definitions \eqref{eq_new_6} and \eqref{eq_new_9}}
\label{sec_equivalences}
In this section, the relationship between Definitions \eqref{eq_new_6} and \eqref{eq_new_9} will be unveiled by providing a direct proof of the equivalence between the two. The proof utilizes the tool we established in Section \ref{sec_lip}. At the end of this section, the $L_p$ affine isoperimetric inequality (with equality condition) is established using the new representation \eqref{eq_new_9} of $L_p$ affine surface area.

Let us first introduce some notations.

Let
\begin{align*}
T_3&=\left\{h:S^{n-1}\rightarrow \mathbb{R}\,\,\mathcal{H}^{n-1} \text{measurable}:0<h<\infty\text{ and } \int_{S^{n-1}}h^n(u)d\mathcal{H}^{n-1}(u)<\infty\right\},\\
T_4&=\left\{h:S^{n-1}\rightarrow \mathbb{R} \text{ continuous}:h>0\right\}. 
\end{align*}
Note that $T_3\supset T_4$.

We claim that for each positive real number $p$ and each $K\in \mathcal{K}_0^n$,
\begin{equation}
\label{eq_local 7}
\begin{aligned}
&\inf_{g\in T_1}\left\{\left(\int_{\partial K}g^{-n}(x)H_K(x)d\mathcal{H}^{n-1}(x)\right)^{\frac{p}{n+p}}\left(\int_{\partial K}g^p(x)(h_K(\nu_K(x)))^{1-p}d\mathcal{H}^{n-1}(x)\right)^{\frac{n}{n+p}}\right\}\\
=&\inf_{g\in T_1}\left\{\left(\int_{H^+}g^{-n}(x)H_K(x)d\mathcal{H}^{n-1}(x)\right)^{\frac{p}{n+p}}\left(\int_{H^+}g^p(x)(h_K(\nu_K(x)))^{1-p}d\mathcal{H}^{n-1}(x)\right)^{\frac{n}{n+p}}\right\}.
\end{aligned}
\end{equation}
Indeed, for each $g\in T_1$, let
\begin{equation*}
g_{\varepsilon}(x)=
\begin{cases}
\varepsilon,&\text{if }x\not \in H^+\text{ and }g(x)>\varepsilon,\\
g(x),&\text{otherwise}.
\end{cases}
\end{equation*}
Notice that $g_\varepsilon\leq \varepsilon$ on $\partial K\setminus H^+$. By \eqref{sharp}, \eqref{astr}, the choice of $g_\varepsilon$, and the fact that $$\int_{\partial K\setminus H^+}h_K(\nu_K(x))^{1-p}d\mathcal{H}^{n-1}(x)$$ is a finite number,
\begin{equation*}
\begin{aligned}
&\lim_{\varepsilon\rightarrow 0}\left(\int_{\partial K}g_{\varepsilon}^{-n}(x)H_K(x)d\mathcal{H}^{n-1}(x)\right)^{\frac{p}{n+p}}\left(\int_{\partial K}g_{\varepsilon}^p(x)(h_K(\nu_K(x)))^{1-p}d\mathcal{H}^{n-1}(x)\right)^{\frac{n}{n+p}}\\
\leq &\left(\int_{H^+}g^{-n}(x)H_K(x)d\mathcal{H}^{n-1}(x)\right)^{\frac{p}{n+p}}\left(\int_{H^+}g^p(x)(h_K(\nu_K(x)))^{1-p}d\mathcal{H}^{n-1}(x)\right.\\
&\phantom{weqweqwesdferwerewrwerwrwerwqwewqewqeqe}\left.+\lim_{\varepsilon\rightarrow 0}\varepsilon^p\int_{\partial K\setminus H^+}(h_K(\nu_K(x)))^{1-p}d\mathcal{H}^{n-1}(x)\right)^{\frac{n}{n+p}}\\
=&\left(\int_{H^+}g^{-n}(x)H_K(x)d\mathcal{H}^{n-1}(x)\right)^{\frac{p}{n+p}}\left(\int_{H^+}g^p(x)(h_K(\nu_K(x)))^{1-p}d\mathcal{H}^{n-1}(x)\right)^{\frac{n}{n+p}}.
\end{aligned}
\end{equation*} 
This shows that the left side is less than or equal to the right side in \eqref{eq_local 7}. By the fact that $H^+\subset \partial K$ and that all integrands in \eqref{eq_local 7} are non-negative, the left side is greater than or equal to the right side in \eqref{eq_local 7}. 

By \eqref{eq_local 1014},  we can further show that
\begin{lemma}
\label{Lemma_H+}
Let $p>0$ be a real number. Suppose $K\in \mathcal{K}_0^n$. We have
\begin{equation*}
\begin{aligned}
&\inf_{g\in{T_2}}\left\{\left(\int_{\partial K}g^{-n}(x)dC_0(K,x)\right)^{\frac{p}{n+p}}\left(\int_{\partial K}g^p(x)(h_K(\nu_K(x)))^{1-p}dC_{n-1}(K,x)\right)^{\frac{n}{n+p}}\right\}\\
=&\inf_{g\in T_1}\left\{\left(\int_{H^+}g^{-n}(x)H_K(x)d\mathcal{H}^{n-1}(x)\right)^{\frac{p}{n+p}}\left(\int_{H^+}g^p(x)(h_K(\nu_K(x)))^{1-p}d\mathcal{H}^{n-1}(x)\right)^{\frac{n}{n+p}}\right\}.
\end{aligned}
\end{equation*}
\end{lemma}

Recall that $F^+$ is the set of unit vectors at which the curvature function $F_K$ is positive (see \eqref{105}). By a similar argument from above,
\begin{equation*}
\begin{aligned}
&\inf_{h\in T_3}\left\{\left(\int_{S^{n-1}}h^n(u)d\mathcal{H}^{n-1}(u)\right)^{\frac{p}{n+p}}\left(\int_{S^{n-1}}h^{-p}(u)h_K^{1-p}(u)F_K(u)d\mathcal{H}^{n-1}(u)\right)^{\frac{n}{n+p}}\right\}\\
=&\inf_{h\in T_3}\left\{\left(\int_{F^+}h^n(u)d\mathcal{H}^{n-1}(u)\right)^{\frac{p}{n+p}}\left(\int_{F^+}h^{-p}(u)h_K^{1-p}(u)F_K(u)d\mathcal{H}^{n-1}(u)\right)^{\frac{n}{n+p}}\right\}.
\end{aligned}
\end{equation*} 

This, together with Hilfssatz 2 and 3 in \cite{leitchtweiss1989} (see also~\cite{leichtweiss1998}), immediately implies
\begin{lemma}
\label{Lemma_F+}
Let $p>0$ be a real number. Suppose $K\in \mathcal{K}_0^n$. We have
\begin{equation*}
\begin{aligned}
&\inf_{h\in T_4}\left\{\left(\int_{S^{n-1}}h^n(u)d\mathcal{H}^{n-1}(u)\right)^{\frac{p}{n+p}}\left(\int_{S^{n-1}}h^{-p}(u)h_K^{1-p}(u)dS_K(u)\right)^{\frac{n}{n+p}}\right\}\\
=&\inf_{h\in T_3}\left\{\left(\int_{F^+}h^n(u)d\mathcal{H}^{n-1}(u)\right)^{\frac{p}{n+p}}\left(\int_{F^+}h^{-p}(u)h_K^{1-p}(u)F_K(u)d\mathcal{H}^{n-1}(u)\right)^{\frac{n}{n+p}}\right\}.
\end{aligned}
\end{equation*}
\end{lemma}

Now, we shall prove the equivalence between Definitions \eqref{eq_new_6} and \eqref{eq_new_9}.

\begin{theorem}
\label{thm_new_3}
Let $p>0$ be a real number. Suppose $K\in \mathcal{K}_0^n$. We have
\begin{equation*}
\begin{aligned}
&\inf_{h}\left\{\left(\int_{S^{n-1}}h^n(u)d\mathcal{H}^{n-1}(u)\right)^{\frac{p}{n+p}}\left(\int_{S^{n-1}}h^{-p}(u)h_K^{1-p}(u)dS_K(u)\right)^{\frac{n}{n+p}}\right\}\\
=&\inf_{g}\left\{\left(\int_{\partial K}g^{-n}(x)dC_0(K,x)\right)^{\frac{p}{n+p}}\left(\int_{\partial K}g^p(x)(h_K(\nu_K(x)))^{1-p}dC_{n-1}(K,x)\right)^{\frac{n}{n+p}}\right\},
\end{aligned}
\end{equation*}
where the infimums are taken over all positive continuous functions $h:S^{n-1}\rightarrow \mathbb{R}$ and $g:\partial K\rightarrow \mathbb{R}$ respectively.
\end{theorem}
\begin{proof}
By Lemmas \ref{Lemma_H+} and \ref{Lemma_F+}, it suffices to show
\begin{equation}
\label{102}
\begin{aligned}
&\inf_{g\in T_1}\left\{\left(\int_{H^+}g^{-n}(x)H_K(x)d\mathcal{H}^{n-1}(x)\right)^{\frac{p}{n+p}}\left(\int_{H^+}g^p(x)(h_K(\nu_K(x)))^{1-p}d\mathcal{H}^{n-1}(x)\right)^{\frac{n}{n+p}}\right\}\\
=&\inf_{h\in T_3}\left\{\left(\int_{F^+}h^n(u)d\mathcal{H}^{n-1}(u)\right)^{\frac{p}{n+p}}\left(\int_{F^+}h^{-p}(u)h_K^{1-p}(u)F_K(u)d\mathcal{H}^{n-1}(u)\right)^{\frac{n}{n+p}}\right\}.
\end{aligned}
\end{equation}

Let us first show that the left side is greater than or equal to the right side in \eqref{102}. 

Let $g\in T_1$ be such that $\int_{H^+}g^{-n}(x)H_K(x)d\mathcal{H}^{n-1}(x)<\infty$ and $r>0$. We first observe that $(\partial K)_r$ and $S^{n-1}$ are $(\mathcal{H}^{n-1},n-1)$ rectifiable and $\mathcal{H}^{n-1}$ measurable (since $(\partial K)_r$ is closed by Lemma \ref{Lemma 4.2}). 

By Lemma \ref{Lemma 4.2}, $\nu_r$, the restriction of $\nu_K$ to $(\partial K)_r$, is Lipschitz and 
\begin{equation}
\label{roman 1}
(\mathcal{H}^{n-1}\mres {(\partial K)_r},n-1)\,\ap\,J\nu_r(x)=H_K(x),
\end{equation}
for $\mathcal{H}^{n-1}$ almost all $x\in (\partial K)_r$.

The fact that $\nu_r$ is Lipschitz and \eqref{sharp} give,
\begin{equation}
\label{eq_local 12}
\mathcal{H}^{n-1}(\nu_r((\partial K)_r\setminus \nor K))=0.
\end{equation}

By \eqref{roman 1} and Federer's coarea formula,
\begin{equation}
\label{eq_local 21}
\begin{aligned}
\int_{(\partial K)_r\cap H^+}g^{-n}(x)H_K(x)d\mathcal{H}^{n-1}(x)
=&\int_{(\partial K)_r}\chara{H^+}(x)g^{-n}(x)H_K(x)d\mathcal{H}^{n-1}(x)\\
=&\int_{S^{n-1}}\left(\int_{(\nu_r)^{-1}(u)}\chara{H^+}(x)g^{-n}(x)d\mathcal{H}^0(x)\right)d\mathcal{H}^{n-1}(u).
\end{aligned}
\end{equation}
It is implied in Federer's coarea formula that 
\begin{equation*}
\int_{(\nu_r)^{-1}(u)}\chara{H^+}(x)g^{-n}(x)d\mathcal{H}^0(x)
\end{equation*}
is $\mathcal{H}^{n-1}$ measurable on $S^{n-1}$ in $u$. By \eqref{hkdifferentiable}, \eqref{eq_local 12}, Lemma \ref{Lemma_reciporical of H and F}, and \eqref{hkdifferentiable}, \eqref{eq_local 12} once again, the following holds for $\mathcal{H}^{n-1}$ almost all $u\in S^{n-1}$,
\begin{equation}
\label{eq_local 22}
\begin{aligned}
\int_{(\nu_r)^{-1}(u)}\chara{H^+}(x)g^{-n}(x)d\mathcal{H}^0(x)
=&\chara{\nu_r((\partial K)_r))\cap D^2(h_K)}{(u)}\chara{H^+}{(\tau_K(u))}g(\tau_K(u))^{-n}\\
=&\chara{\nu_r((\partial K)_r)\cap \nor K)\cap D^2(h_K)}{(u)}\chara{H^+}{(\tau_K(u))}g(\tau_K(u))^{-n}\\
=&\chara{\nu_r((\partial K)_r)\cap \nor K)\cap D^2(h_K)}{(u)}\chara{F^+}{(u)}g(\tau_K(u))^{-n}\\
=&\chara{\nu_r((\partial K)_r)}{(u)}\chara{F^+}{(u)}g(\tau_K(u))^{-n}.
\end{aligned}
\end{equation}
Hence $\chara{\nu_r((\partial K)_r)}(u)\chara{F^+}(u)(g(\tau_K(u)))^{-n}$, as a function of $u$, is $\mathcal{H}^{n-1}$ measurable on $S^{n-1}$. By \eqref{eq_local 21} and \eqref{eq_local 22},
\begin{equation}
\label{103}
\int_{(\partial K)_r\cap H^+}g^{-n}(x)H_K(x)d\mathcal{H}^{n-1}(x)
=\int_{S^{n-1}}\chara{\nu_r((\partial K)_r)}(u)\chara{F^+}(u)(g(\tau_K(u)))^{-n}d\mathcal{H}^{n-1}(u).
\end{equation}
Note that $(\partial K)_r$ is increasing as $r\rightarrow 0$. Let $r\rightarrow 0$ in \eqref{103}. It follows from the monotone convergence theorem, Lemma \ref{Lemma 4.1}, and Lemma \ref{Lemma 4.3} that, 
\begin{equation}
\label{eq_local 8}
\int_{H^+}g^{-n}(x)H_K(x)d\mathcal{H}^{n-1}(x)=\int_{F^+}(g(\tau_K(u)))^{-n}d\mathcal{H}^{n-1}(u).
\end{equation}

With the same technique, one can prove that
\begin{equation}
\label{eq_local 14}
\int_{H^+}g^p(x)(h_K(\nu_K(x)))^{1-p}d\mathcal{H}^{n-1}(x)=\int_{F^+}(g(\tau_K(u)))^ph_K^{1-p}(u)F_K(u)d\mathcal{H}^{n-1}(u).
\end{equation}

Set 
\begin{equation*}
h(u)=
\begin{cases}
1/g(\tau_K(u)),& \text{if } u\in F^+,\\
1,& \text{otherwise}. 
\end{cases}
\end{equation*}
Note that $0<h<\infty$ and $h$ is $\mathcal{H}^{n-1}$ measurable on $S^{n-1}$. By \eqref{eq_local 8}, $h^n$ is $\mathcal{H}^{n-1}$ integrable on $S^{n-1}$. Hence $h\in T_3$. By \eqref{eq_local 8}, \eqref{eq_local 14}, the left side is greater than or equal to the right side in \eqref{102}.

Now we show that the left side is less than or equal to the right side in \eqref{102}. \\
\indent Let $h\in T_3$ be such that $\int_{F^+}h^{-p}(u)h_K^{1-p}(u)F_K(u)d\mathcal{H}^{n-1}(u)<\infty$ and $i\in \mathbb{N}^+$. We first observe that $A_i$ and $\partial K$ are $(\mathcal{H}^{n-1}, n-1)$ rectifiable and $\mathcal{H}^{n-1}$ measurable (since $A_i$ is closed by Lemma \ref{lemma_closedness}). By Lemma \ref{Lemma_tau lip}, $\tau_i$, the restriction of $\tau_K$ to $A_i$, is Lipschitz. By Lemma \ref{Lemma_Jacobian}, for $\mathcal{H}^{n-1}$ almost all $u\in A_i$,
\begin{equation}
\label{roman ii}
(\mathcal{H}^{n-1}\mres {A_i},n-1)\,\ap\,J\tau_i(u)=F_K(u).
\end{equation}
\indent The fact that $\tau_i$ is Lipschitz and \eqref{hkdifferentiable} give,
\begin{equation}
\label{eq_local 16}
\mathcal{H}^{n-1}(\tau_i(A_i\setminus D^2(h_K)))=0.
\end{equation}
\indent By \eqref{roman ii} and Federer's coarea formula, 
\begin{equation}
\label{eq_local 24}
\begin{aligned}
\int_{A_i\cap F^+}h^{-p}(u)h_K^{1-p}(u)F_K(u)d\mathcal{H}^{n-1}(u)
=&\int_{A_i}\chara{F^+}(u)h^{-p}(u)h_K^{1-p}(u)F_K(u)d\mathcal{H}^{n-1}(u)\\
=&\int_{\partial K}\left(\int_{(\tau_i)^{-1}(x)}\chara{F^+}(u)h^{-p}(u)h_K^{1-p}(u)d\mathcal{H}^0(u)\right)d\mathcal{H}^{n-1}(x).
\end{aligned}
\end{equation}
\indent It is implied in Federer's coarea formula that 
\begin{equation*}
\int_{(\tau_i)^{-1}(x)}\chara{F^+}(u)h^{-p}(u)h_K^{1-p}(u)d\mathcal{H}^0(u)
\end{equation*}
is $\mathcal{H}^{n-1}$ measurable on $\partial K$ in $x$. By \eqref{sharp}, \eqref{eq_local 16}, Lemma \ref{Lemma_reciporical of H and F} and \eqref{sharp}, \eqref{eq_local 16} once again, the following holds for $\mathcal{H}^{n-1}$ almost all $x\in\partial K$,
\begin{equation}
\label{eq_local 25}
\begin{aligned}
\int_{(\tau_i)^{-1}(x)}\chara{F^+}(u)h^{-p}(u)h_K^{1-p}(u)d\mathcal{H}^0(u)
=&\chara{\tau_i(A_i)\cap \nor K}{(x)}\chara{F^+}{(\nu_K(x))}h(\nu_K(x))^{-p}h_K(\nu_K(x))^{1-p}\\
=&\chara{\tau_i(A_i\cap D^2(h_K))\cap \nor K}{(x)}\chara{F^+}{(\nu_K(x))}h(\nu_K(x))^{-p}h_K(\nu_K(x))^{1-p}\\
=&\chara{\tau_i(A_i\cap D^2(h_K))\cap \nor K}{(x)}\chara{H^+}{(x)}h(\nu_K(x))^{-p}h_K(\nu_K(x))^{1-p}\\
=&\chara{\tau_i(A_i)}{(x)}\chara{H^+}{(x)}h(\nu_K(x))^{-p}h_K(\nu_K(x))^{1-p}.
\end{aligned}
\end{equation}
Hence $\chara{\tau_i(A_i)}(x)\chara{H^+}(x)(h(\nu_K(x)))^{-p}(h_K(\nu_K(x)))^{1-p}$, as a function of $x$, is $\mathcal{H}^{n-1}$ measurable on $\partial K$. By \eqref{eq_local 24} and \eqref{eq_local 25},
\begin{equation}
\label{104}
\int_{A_i\cap F^+}h^{-p}(u)h_K^{1-p}(u)F_K(u)d\mathcal{H}^{n-1}(u)=\int_{\partial K}\chara{\tau_i(A_i)}(x)\chara{H^+}(x)(h(\nu_K(x)))^{-p}(h_K(\nu_K(x)))^{1-p}d\mathcal{H}^{n-1}(x).
\end{equation}
Note that $A_i$ is increasing as $i\rightarrow \infty$. Let $i\rightarrow \infty$ in \eqref{104}. It follows from the monotone convergence theorem, Corollary \ref{corollary_a.e.}, and Lemma \ref{lemma_H>0} that
\begin{equation}
\label{eq_local 10}
\int_{F^+}h^{-p}(u)h_K^{1-p}(u)F_K(u)d\mathcal{H}^{n-1}(u)=\int_{H^+}(h(\nu_K(x)))^{-p}(h_K(\nu_K(x)))^{1-p}d\mathcal{H}^{n-1}(x).
\end{equation}

With the same technique, one can prove
\begin{equation}
\label{eq_local 17}
\int_{F^+}h^{n}(u)d\mathcal{H}^{n-1}(u)=\int_{H^+}(h(\nu_K(x)))^{n}H_K(x)d\mathcal{H}^{n-1}(x).
\end{equation}

Set 
\begin{equation*}
g(x)=
\begin{cases}
1/h(\nu_K(x)),& \text{if } x\in H^+,\\
1,& \text{otherwise}. 
\end{cases}
\end{equation*}
Note that $0<g<\infty$ and $g$ is $\mathcal{H}^{n-1}$ measurable. By \eqref{eq_local 10}, $g^p(h_K\circ\nu_K)^{1-p}$ is $\mathcal{H}^{n-1}$ integrable on $\partial K$. Hence $g\in T_1$. By \eqref{eq_local 10}, \eqref{eq_local 17}, the left side is less than or equal to the right side in \eqref{102}.
\end{proof}

The proof given above reveals that Definition \eqref{eq_new_9} is ``polar'' to \eqref{eq_new_6} and they are linked by the Gauss map and the inverse Gauss map.

The $L_p$ affine isoperimetric inequality, which is the extension of the affine isoperimetric inequality \eqref{star} of affine differential geometry, was first established by Lutwak in \cite{lutwak1996}. Thanks to \eqref{eq_new_15}, we are ready to prove the $L_p$ affine isoperimetric inequality using the new representation \eqref{eq_new_9}. 

\begin{theorem}
{\label{prop_new_2}
Let $p>0$ be a real number. Suppose $K\in \mathcal{K}_0^n$ and has the origin as its centroid. We have,
\begin{equation}
\label{eq_new_16}
\Omega_p(K)\leq n \omega_n^{\frac{2p}{n+p}}V(K)^{\frac{n-p}{n+p}},
\end{equation}
with equality if and only if $K$ is an ellipsoid.
}
\end{theorem}
\begin{proof}
Recall from \eqref{eq_new_15} that, for each convex body $L$ that has the origin as its centroid,
\begin{equation}
\label{eq_new_17}
\frac{1}{n}\Omega_p(K)\leq W(K,h_L\circ\nu_K)^{\frac{p}{n+p}}V_p(K,L)^{\frac{n}{n+p}}.
\end{equation}
Based on the change of variable formula established in the proof of Theorem \ref{thm_new_3},
\begin{equation}
\label{eq_new_18}
W(K,h_L\circ\nu_K)\leq V(L^*),
\end{equation}
where $L^*$ is the polar body of $L$. Combining \eqref{eq_new_17} and \eqref{eq_new_18}, we get,
\begin{equation}
\label{eq_new_19}
\frac{1}{n}\Omega_p(K)\leq V(L^*)^{\frac{p}{n+p}}V_p(K,L)^{\frac{n}{n+p}}.
\end{equation}

It was observed in \cite{lutwak1991,lutwak1996} that \eqref{eq_new_19} is stronger  than \eqref{eq_new_16}. Indeed, just by replacing $L$ in \eqref{eq_new_19} by $K$ and using the Blaschke-Santal\'{o} inequality, we get \eqref{eq_new_16}.

Equality holds only if the equality holds for the Blaschke-Santal\'{o} inequality, i.e., $K$ is an ellipsoid. That the equality does hold for ellipsoids follows from a direct calculation. 
\end{proof}

\section{A Further Application of Section \ref{sec_lip}}
Hug in~\cite{daniel1996} proved the equivalence between Definitions \eqref{eq_new_7} and \eqref{eq_new_8} by applying Federer's coarea formula to the Lipschitz map $\nu_r$, the restriction of $\nu_K$ to $(\partial K)_r$. Here we provide a dual proof of the same result by applying Federer's coarea formula to the Lipschitz map $\tau_i$, the restriction of $\tau_K$ to $A_i$, which is made possible by the discussion in Section \ref{sec_lip}.

\begin{theorem}
Let $p>0$ be a real number. Suppose $K\in \mathcal{K}_0^n$. We have
\begin{equation*}
\int_{S^{n-1}}\left(\frac{F_K(u)}{h_K^{p-1}(u)}\right)^{\frac{n}{n+p}}d\mathcal{H}^{n-1}(u)=\int_{\partial K}\left(\frac{H_K(x)}{(h_K(\nu_K(x)))^{(p-1)\frac{n}{p}}}\right)^{\frac{p}{n+p}}d\mathcal{H}^{n-1}(x).
\end{equation*}
\end{theorem}
\ifdoublespace
\vspace{-0.2in}
\fi
\begin{proof}
For each $i\in \mathbb{N}^+$, observe that $A_i$ and $\partial K$ are $(\mathcal{H}^{n-1}, n-1)$ rectifiable and $\mathcal{H}^{n-1}$ measurable (since $A_i$ is closed by Lemma \ref{lemma_closedness}). 

By Lemmas \ref{Lemma_tau lip}, \ref{Lemma_Jacobian}, Federer's coarea formula, \eqref{sharp}, \eqref{eq_local 16}, Lemma \ref{Lemma_reciporical of H and F}, and \eqref{eq_local 16} once again, we have
\begin{equation}
\label{106}
\begin{aligned}
\int_{A_i}\left(\frac{F_K(u)}{h_K^{p-1}(u)}\right)^{\frac{n}{n+p}}d\mathcal{H}^{n-1}(u)=&\int_{A_i\cap F^+}\left(\frac{F_K(u)}{h_K^{p-1}(u)}\right)^{\frac{n}{n+p}}d\mathcal{H}^{n-1}(u)\\
=&\int_{A_i}\chara{F^+}(u)\left(\frac{F_K^{-\frac{p}{n}}(u)}{h_K^{p-1}(u)}\right)^{\frac{n}{n+p}}F_K(u)d\mathcal{H}^{n-1}(u)\\
=&\int_{\partial K}\left(\int_{\tau_{i}^{-1}(x)}\chara{F^+}(u)\left(\frac{F_K^{-\frac{p}{n}}(u)}{h_K^{p-1}(u)}\right)^{\frac{n}{n+p}}d\mathcal{H}^0(u)\right)d\mathcal{H}^{n-1}(x)\\
=&\int_{\tau_i(A_i)\cap \nor K}\chara{F^+}(\nu_K(x))\left(\frac{(F_K(\nu_K(x)))^{-\frac{p}{n}}}{(h_K(\nu_K(x)))^{p-1}}\right)^{\frac{n}{n+p}}d\mathcal{H}^{n-1}(x)\\
=&\int_{\tau_i(A_i\cap D^2(h_K))\cap \nor K}\chara{F^+}(\nu_K(x))\left(\frac{(F_K(\nu_K(x)))^{-\frac{p}{n}}}{(h_K(\nu_K(x)))^{p-1}}\right)^{\frac{n}{n+p}}d\mathcal{H}^{n-1}(x)\\
=&\int_{\tau_i(A_i\cap D^2(h_K))\cap \nor K}\chara{H^+}(x)\left(\frac{H_K(x)}{(h_K(\nu_K(x)))^{(p-1)\frac{n}{p}}}\right)^{\frac{p}{n+p}}d\mathcal{H}^{n-1}(x)\\
=&\int_{\tau_i(A_i)\cap \nor K}\left(\frac{H_K(x)}{(h_K(\nu_K(x)))^{(p-1)\frac{n}{p}}}\right)^{\frac{p}{n+p}}d\mathcal{H}^{n-1}(x).\end{aligned}
\end{equation}

Now, let $i\rightarrow \infty$ in \eqref{106}. It follows from the monotone convergence theorem, Corollary \ref{corollary_a.e.}, and Lemma \ref{lemma_H>0} that
\begin{equation*}
\begin{aligned}
\int_{S^{n-1}}\left(\frac{F_K(u)}{h_K^{p-1}(u)}\right)^{\frac{n}{n+p}}d\mathcal{H}^{n-1}(u)&=\int_{H^+}\left(\frac{H_K(x)}{(h_K(\nu_K(x)))^{(p-1)\frac{n}{p}}}\right)^{\frac{p}{n+p}}d\mathcal{H}^{n-1}(x)\\
&=\int_{\partial K}\left(\frac{H_K(x)}{(h_K(\nu_K(x)))^{(p-1)\frac{n}{p}}}\right)^{\frac{p}{n+p}}d\mathcal{H}^{n-1}(x).
\end{aligned}
\end{equation*}

\end{proof}
\thanks{This material is based upon work supported in part by the National Science Foundation under Grant DMS-1312181.}

\end{document}